\documentclass[twocolumn]{autart}%
\usepackage{bbm}
\usepackage{amssymb}
\usepackage{amsmath}
\usepackage{amssymb}
\usepackage{mathrsfs}
\usepackage{amsfonts}
\usepackage{graphicx}
\usepackage{color}
\usepackage{graphicx}
\usepackage{wrapfig}
\usepackage{epstopdf}%
\providecommand{\U}[1]{\protect \rule{.1in}{.1in}}
\allowdisplaybreaks[4]
\newenvironment{proof}[1][Proof]{\noindent \textbf{#1.} }{\  \rule{0.5em}{0.5em}}
\newtheorem{remark}{Remark}
\newtheorem{theorem}{Theorem}
\newtheorem{lemma}{Lemma}

\newtheorem{definition}{Definition}
\newtheorem{corollary}{Corollary}

\newtheorem{assumption}{Assumption}
\begin{document}
\begin{frontmatter}
\title{Using artificial delays for prescribed/fixed-time static output feedback control of linear systems \thanksref{footnoteinfo}}
\vspace{-0.8cm}
\thanks[footnoteinfo]{This work was supported in part by Azrieli International Postdoctoral Fellowship, in part by the Israel Science Foundation
under Grant 446/24, and in part by the Chana and Heinrich Manderman
Chair at Tel Aviv University. \\$^{\ast}$ Corresponding author.}
\author{Kang-Kang Zhang}\ead{kangkang\_kkz@163.com}, \quad
\author{Emilia Fridman} $^{\ast}$\ead{emilia@tauex.tau.ac.il}\quad
\address{School of Electrical and Computer Engineering, Tel Aviv University, 6997801, Israel}
\vspace{-0.8cm}
\begin{abstract}
In this paper, we address the prescribed/fixed-time control problem for linear time-invariant systems with relative degree $r\geq 2$ using delayed time-varying static output feedback. In addition to controllability and observability, we assume that a truncated observability matrix is nonsingular.
Within the act-and-wait control framework, two types of artificial delays are introduced into the output feedback: one to realize static output feedback and the other to achieve prescribed/fixed-time control. By constructing smooth periodic delayed output feedback controllers, we achieve prescribed-time input-to-state stability (ISS) with respect to matched disturbances and measurement noise. Next, we consider systems with measurement delays. Note that existing predictor methods are not applicable to static output feedback delay compensation. We propose delayed controllers that achieve fixed-time stability in the presence of large known slowly-varying delays (with delay derivatives less than one) and even in the presence of large unknown constant delays for systems with nilpotent matrices. Moreover, we present a design for the case of discrete-time periodic measurements with a small sampling period. When the system reduces to a chain of integrators, prescribed/fixed-time convergence can be achieved with arbitrarily large artificial delays introduced by static output feedback avoiding high gain and managing with large known slowly-varying or sawtooth delays (corresponding to discrete-time measurements) or unknown constant delays. In the case of unknown constant delays, this leads to a low-gain controller that may efficiently attenuate measurement noise. Numerical examples verify the effectiveness of the proposed method.
\end{abstract}
\end{frontmatter}
\vspace{-0.8cm}

\section{Introduction}

Since the derivative of the output can be approximated by finite differences,
i.e., $\dot{y}(t)\approx(y(t)-y(t-d))/d$, static output feedback with delays
can stabilize linear time-invariant systems that can be stabilized by a
controller depending on the derivative of the output, which cannot be measured
directly \cite{Boussaada26icsl,emiliabook,emilia16auto,emilia25,Gomez19auto,kara08scl,nicule04tac,ramirez2015design}. In
the existing works it is proved that the resulting delayed static
output-feedback stabilizes the system for small enough delay $d$. By using the
frequency-domain method, the output feedback stabilization problem for a chain
of integrators with multiple delays was solved in \cite{nicule04tac}. By employing
Taylor's expansion of the delayed terms with the remainders in the integral
form and constructing appropriate Lyapunov--Krasovskii functionals, simple and
efficient LMIs for stability of the static output feedback with delays were
obtained in \cite{emilia17auto}. Further, improved reduced-order LMIs were presented for linear systems that
can be stabilized by derivative-dependent feedback
\cite{anton18auto}. However, the above results can only achieve exponential stability, which may not necessarily meet the requirements of
practical engineering for faster convergence rates.

Compared with asymptotic stability or exponential stability,
finite/fixed/prescribed-time control has long been a research
hotspot in the control community due to its faster convergence rate
\cite{bhat00siam,mazenc25tac,ning23auto,polyakov11tac,song17auto,ye25auto}. For the chain of integrators,
accelerated or practical fixed-time input-to-state
stability (ISS) using static homogeneous output feedback was studied in
\cite{Efimov20auto,Nekhoroshikh22auto}, where ISS characterizes the system's robustness
against external inputs and convergence to a small residual set, and the term
``practical'' means that the state converges to a small neighborhood of the
origin rather than to the origin itself, even in the absence of external
inputs.

%By periodically activating and deactivating delayed feedback, the
%act-and-wait control strategy can be used to achieve finite-time convergence
%\cite{Insperger06tcst,kara06siam,Stepan06arc}.
%The act-and-wait control strategy, which periodically activates and
%deactivates delayed feedback, was formulated and systematically investigated
%in \cite{Insperger06tcst,Stepan06arc}.
For stabilization, the act-and-wait control strategy, which periodically activates and
deactivates delayed feedback, was proposed and investigated in
\cite{Insperger06tcst,Stepan06arc}. More recently, building on the
same principle, a smooth periodic delayed-feedback control method was
developed for general controllable linear systems with a known input delay,
where the convergence time is determined by an artificial delay
\cite{zhou22tac}.
In the latter paper, it was demonstrated that in the presence of an input delay
$\tau>0$, the convergence time $T$ cannot be less than $2\tau$ (i.e.,
fixed-time control), while in the absence of such a delay, the convergence time
$T$ can be arbitrarily prescribed (i.e., prescribed-time control). Subsequently, by employing the act-and-wait strategy,
fixed-time stabilization by state feedback was achieved in
\cite{zhang26tac} for systems with nilpotent matrices in the
presence of unknown constant measurement delays. In this paper, we refer to fixed-time stability (respectively, fixed-time ISS) whose
convergence time $T>0$ can be arbitrarily prescribed as prescribed-time
stability (respectively, prescribed-time ISS).

In the present paper, we address prescribed/fixed-time control of linear systems
by introducing artificial delays into time-varying static output feedback
within an act-and-wait framework. To the best of our knowledge, this provides
the first static output-feedback solution to prescribed/fixed-time stabilization
of linear systems. In contrast, the existing static output-feedback methods
mainly establish asymptotic, exponential, accelerated, or practical
fixed-time stability
\cite{Efimov20auto,emilia17auto,Nekhoroshikh22auto,anton18auto}.
For linear time-invariant systems, the proposed controller further guarantees
prescribed-time input-to-state stability with respect to matched disturbances
and measurement noise. Rather than imposing the sufficiently small delay restrictions arising from
Lyapunov--Krasovskii-functional-based LMI conditions
\cite{Efimov20auto,emilia17auto,Nekhoroshikh22auto,anton18auto}, the admissibility of the artificial delays is characterized by
an algebraic invertibility condition. This condition is naturally satisfied for all sufficiently small delays,
and is therefore consistent with the existing small-delay results. More importantly, the artificial delays are not required to remain small
and only need to avoid the values at which the corresponding matrix becomes
singular.

The existing static output-feedback methods considered in
\cite{Efimov20auto,emilia16auto,emilia17auto,Nekhoroshikh22auto,anton18auto}
are developed for delay-free output measurements or small ones. To address this
limitation, the proposed framework is extended to systems either with large known slowly-varying output delays (with delay derivatives less than one), or with small sawtooth output delays corresponding to discrete-time output measurements, or with large unknown constant output delays when the system matrix is
nilpotent. In particular, for chains of integrators, prescribed/fixed-time convergence
can be achieved with arbitrarily large artificial delays avoiding high gain and managing with large known slowly-varying or sawtooth delays (corresponding to discrete-time measurements) or unknown constant delays, thereby removing
the small-delay restrictions in \cite{Efimov20auto,emilia17auto,Nekhoroshikh22auto,anton18auto}.
In the case of unknown constant delays, this leads to a low-gain controller that may efficiently attenuate measurement noise. In addition, unlike \cite{nicule04tac}, which stabilizes chains of integrators using delayed feedback with a slow convergence rate, the proposed method achieves rapid prescribed-time
convergence.

The remainder of this paper is organized as follows. For linear
time-invariant systems without measurement delay, the prescribed-time ISS
results with respect to matched disturbances and measurement noise are studied
in Section~\ref{sec2}. For systems with measurement delay, fixed-time results
are proposed in Section~\ref{sec3}. Some numerical examples are given in
Section~\ref{sec4} to illustrate the proposed theory, and Section~\ref{sec5}
concludes this paper.

\textbf{Notation} $\mathbb{Z}_{0}^{+}$ denotes the set of nonnegative integers, and
$\mathbb{N}$ denotes the set of positive integers.
$|\cdot|$ denotes the usual Euclidean norm on
$\mathbb{R}^{n}$ and the corresponding induced matrix norm.
$\lfloor a\rfloor$ denotes the integer part of $a$.
$\mathbf{C}^{i}$ denotes the class of $i$ times continuously
differentiable functions. For a matrix $A$,
$\sigma_{\min}(A)$ denotes its smallest singular value. For $i=0,1,\ldots,r-1$, $e_{p,i}\in\mathbb{R}^{rp\times p}$
denotes the block selection matrix whose $(i+1)$-th block is
$I_p$ and whose remaining blocks are zero.

\section{\label{sec2}Prescribed-Time Static Output Feedback ISS}

\subsection{Systems with an Invertible Truncated Observability Matrix}

Consider the linear system with matched disturbances
\begin{equation}
\left\{
\begin{array}
[c]{rl}%
\dot{x}(t) & =Ax(t)+B(u(t)+w_{1}(t)),\\
y(t) & =Cx(t)+w_{2}(t),
\end{array}
\right.  \label{sys_2}%
\end{equation}
where $x(t) \in%
%TCIMACRO{\U{211d} }%
%BeginExpansion
\mathbb{R}
%EndExpansion
^{n}$, $u(t) \in%
%TCIMACRO{\U{211d} }%
%BeginExpansion
\mathbb{R}
%EndExpansion
^{m}$, $y(t) \in%
%TCIMACRO{\U{211d} }%
%BeginExpansion
\mathbb{R}
%EndExpansion
^{p}$, and relative degree $r\geq2$, i.e.,
\begin{equation}
CA^{i}B=0,\quad i=0,1,\ldots,r-2,\quad CA^{r-1}B\neq0, \label{cab}%
\end{equation}
$w_{1}(t)\in\mathbb R^m$ and $w_{2}(t)\in\mathbb R^p$ are piecewise continuous disturbances satisfying
$\vert w_{1}(t)\vert\leq w_{1\infty}$ and $\vert w_{2}(t)\vert\leq w_{2\infty
}$ with $w_{1\infty}$ and $w_{2\infty}$ being two nonnegative constants.

\begin{assumption}
\label{ass1}Assume that $n=pr$ and $(A,B)$ is controllable.
\end{assumption}

Denote by $y_{0}(t)=Cx(t)$ the nominal (non-disturbed) output. Clearly, in
view of (\ref{sys_2}) and (\ref{cab}), we can calculate
\begin{equation}
y_{0}^{\left(  i\right)  }\left(  t\right)  =CA^{i}x\left(  t\right)  ,\quad
i=1,\ldots,r-1. \label{y_norm_def}%
\end{equation}
Then the non-disturbed system (\ref{sys_2}) with $w_{1}\left(  t\right)  =0$
can be stabilized by controller
\begin{equation}
u_{0}(t)=\sum_{i=0}^{r-1}\bar{K}_{i}y_{0}^{(i)}(t), \label{u_normal}%
\end{equation}
where $\bar{K}_{i}\in\mathbf{%
%TCIMACRO{\U{211d} }%
%BeginExpansion
\mathbb{R}
%EndExpansion
}^{m\times p}$ is the constant controller gain. The derivatives of $y_{0}(t)$
can be approximated by finite-differences $\tilde{y}_{i}(t) \approx
y_{0}^{(i)}(t)$:
\begin{align}
\tilde{y}_{0}(t)  &  =y_{0}(t),\label{y_0_hat}\\
\tilde{y}_{i}(t)  &  =\frac{\tilde{y}_{i-1}(t)-\tilde{y}_{i-1}(t-d)}%
{d}\nonumber\\
&  =\frac{1}{d^{i}}\sum_{k=0}^{i}%
\begin{pmatrix}
i\\
k
\end{pmatrix}
(-1)^{k}y_{0}(t-kd), \label{y_i_hat}%
\end{align}
with $i=1,2,\ldots,r-1$ and a delay $d>0$. Therefore, motivated by the approximation
$\tilde{y}_{i}(t)\approx y_{0}^{(i)}(t)$, the derivative-dependent
controller (\ref{u_normal}) can be implemented in the following
delay-dependent form:
\begin{equation}
u_{0}(t)=\sum_{i=0}^{r-1}\bar{K}_{i}\tilde{y}_{i}(t)
=\sum_{i=0}^{r-1}K_{i}y_{0}(t-id), \label{u_normal2}
\end{equation}
where we set $y_{0}(t)=y_{0}(0)$ for $t<0$, $K_i\in\mathbb{R}^{m\times p}$, $i=0,1,\ldots,r-1$, are the
block columns of $K=[K_{0},K_{1},\ldots,K_{r-1}]\in\mathbb{R}^{m\times rp}$.
More precisely, let $\bar{K}=[\bar{K}_{0},\bar{K}_{1},\ldots,\bar{K}_{r-1}]
\in\mathbb{R}^{m\times rp}$.
Then $K=\bar{K}M^{\mathrm{T}}$
where $M\in\mathbb{R}^{rp\times rp}$ is the upper triangular block matrix
with block entries
\begin{equation}
M_{ij}=\left\{
\begin{array}{ll}
(-1)^{i}\binom{j}{i}\dfrac{1}{d^{j}}I_{p}, & j\geq i,\\[1.5mm]
0_{p}, & j<i,
\end{array}
\right.
\quad i,j=0,1,\ldots,r-1. \label{Mij}
\end{equation}
If the nominal system (\ref{sys_2}) can be stabilized by the
derivative-dependent control (\ref{u_normal}), then it can be stabilized by
the delay-dependent control (\ref{u_normal2}) with small enough delays
(\cite{Fre09siam,anton18auto}). The following lemma will be used in the analysis.

\begin{lemma}
\cite{anton18auto} If $y_{0}(t)\in\mathbf{C}^{i}$ and $y_{0}^{(i)}(t)$ is
absolutely continuous with $i\in%
%TCIMACRO{\U{2115} }%
%BeginExpansion
\mathbb{N}
%EndExpansion
$, then $\tilde{y}_{i}(t)$ defined in (\ref{y_i_hat}) satisfies
\begin{equation}
\tilde{y}_{i}(t) = y_{0}^{(i)}(t) -\int_{t-id}^{t}\varphi_{i}(t-s)
y_{0}^{(i+1)}(s) \mathrm{d}s, \label{y_h}%
\end{equation}
where $\varphi_{1}(\xi)=\frac{d-\xi}{d}$ and for $i\in%
%TCIMACRO{\U{2115} }%
%BeginExpansion
\mathbb{N}
%EndExpansion
,$
\[
\varphi_{i+1}(\xi)=\left\{
\begin{array}
[c]{ll}%
\int_{0}^{\xi}\frac{\varphi_{i}(\lambda)}{d}\mathrm{d}\lambda+\frac{d-\xi}%
{d}, & \xi\in\lbrack0,d],\\
\int_{\xi-d}^{\xi}\frac{\varphi_{i}(\lambda)}{d}\mathrm{d}\lambda, & \xi
\in(d,id),\\
\int_{\xi-d}^{id}\frac{\varphi_{i}(\lambda)}{d}\mathrm{d}\lambda, & \xi
\in\lbrack id,id+d].
\end{array}
\right.
\]

\end{lemma}

We make the following assumption.

\begin{assumption}
\label{ass2}%
Assume that
\begin{equation}
W_{\mathrm{o}}=[C^{\mathrm{T}},(CA)^{\mathrm{T}},\ldots,(CA^{r-1}%
)^{\mathrm{T}}]^{\mathrm{T}}\in\mathbf{%
%TCIMACRO{\U{211d} }%
%BeginExpansion
\mathbb{R}
%EndExpansion
}^{n\times n}, \label{Woo}%
\end{equation}
is invertible.
\end{assumption}
\begin{remark}
If $W_{\mathrm{o}}$ is invertible, then necessarily $n=pr$, and
$
[y_0^{\mathrm T},\dot y_0^{\mathrm T},\ldots,
(y_0^{(r-1)})^{\mathrm T}]^{\mathrm T}
=W_{\mathrm{o}}x
$
defines a nonsingular transformation of the state. Hence, the state $x(t)$ can be
uniquely reconstructed from the output $y_0(t)$ and its derivatives up to order
$r-1$. In this sense, with respect to the chosen output $y(t)=Cx(t)$, the system has
no internal dynamics in the input-output normal form; see Section 2.1 of
\cite{zhou23auto}. This condition plays an important role in the proposed delayed static
output-feedback fixed/prescribed-time control design. Indeed, if internal
dynamics exist, then the output and its finite-order derivatives may not
directly provide a complete characterization of the whole state, and the
present static output-feedback construction may not, in general, directly
ensure that all state components converge to zero within a
finite/fixed/prescribed time. If $W_{\mathrm{o}}$ is not invertible, one
possible approach is to redefine or augment the measured output, when
feasible, so that the corresponding matrix $W_{\mathrm{o}}$ associated with
the new output $\tilde y(t)=\tilde Cx(t)$ becomes invertible.
\end{remark}
Let $h>0$ be an artificial delay. In this paper, following \cite{zhou22tac},
where full state-feedback was considered, to achieve
prescribed/fixed-time control, we design time-varying gains, say,
$K=K(t)$ and $\bar{K}=\bar{K}(t)$.

\begin{definition}
\label{def1}$S(t):\mathbb{R}_{\geq0}\rightarrow\mathbb{R}$ is a smooth function satisfying

\begin{itemize}
\item $S(t)$ is smooth and $3h$-periodic;

\item $S(t)\geq0$, for all $ t\in\lbrack2h,3h]$, and there exists an $h_{\ast}%
\in(2h,3h)$ such that $S(h_{\ast})>0$;

\item $S(2h)=0$ and $S(3h)=0$.
\end{itemize}
\end{definition}

For $l\in\mathbb{Z}_{0}^{+}$, define
\begin{equation}
g(t)= \left\{
\begin{array}
[c]{cl}%
0, & t\in[3lh,(3l+2)h],\\
1, & t\in((3l+2)h,3(l+1)h).
\end{array}
\right. \label{g_t}%
\end{equation}
Let
\begin{equation}
\bar{K}(t) = -S(t)B^{\mathrm{T}}\mathrm{e}^{A^{\mathrm{T}}(3lh-t)}
W_{\mathrm{c}}\mathrm{e}^{-A(t-h-3lh)} (W_{\mathrm{o}}+L(d))^{-1}%
,\label{K_hat}%
\end{equation}
where \begin{equation}
\bar{K}(t)= [\bar{K}_{0}(t),\bar{K}_{1}(t),\ldots,\bar{K}_{r-1}(t)]
\in\mathbb{R}^{m\times n},
\label{Kbar_block_def}
\end{equation} $W_{\mathrm{o}}$ is defined in (\ref{Woo}), and
\begin{equation}
W_{\mathrm{c}} = \left(  \int_{2h}^{3h} \mathrm{e}^{-As}BS(s)B^{\mathrm{T}}
\mathrm{e}^{-A^{\mathrm{T}}s}\mathrm{d}s \right) ^{-1}.\label{W_c_def}%
\end{equation}
Moreover,
\begin{equation}
L(d)= [0_{p\times n}^{\mathrm{T}}, L_{1}^{\mathrm{T}}(d),\ldots,
L_{r-1}^{\mathrm{T}}(d)]^{\mathrm{T}} \in\mathbb{R}^{n\times n},
\label{L_h_whole_def}
\end{equation}
where for $i=0,1,\ldots,r-1$,
\begin{equation}
L_{i}(d) = CA^{i+1}\int_{id}^{0}\varphi_{i}(u) \mathrm{e}^{-Au}\mathrm{d}%
u,\label{L_h_def}%
\end{equation}
with $\varphi_{0}\equiv1$.

We first prove that $W_{\mathrm{c}}$ is invertible and that there exists a
sufficiently small positive constant $d_{\ast}$ such that for $d\in(0,d_{\ast
})$, $W_{\mathrm{o}}+L(d)$ is invertible. Regarding $W_{\mathrm{c}}$, the
controllability of $(A,B)$ together with Theorem 1 in \cite{zhou22tac} implies
that $W_{\mathrm{c}}$ is invertible. Moreover, since
\begin{align}
&  \lim_{d\rightarrow0^{+}}\int_{id}^{0}\varphi_{i}(u)
\mathrm{e}^{-Au}\mathrm{d}u=0, \label{hhh}%
\end{align}
there exists a sufficiently small constant $d_{\ast}$ such that for
$d\in(0,d_{\ast})$, $I_{n}+W_{\mathrm{o}}^{-1}L(d)$ is invertible. Hence, for
$d\in(0,d_{\ast})$, the matrix $W_{\mathrm{o}}+L(d)=W_{\mathrm{o}}%
(I_{n}+W_{\mathrm{o}}^{-1}L(d))$ is invertible.

\begin{remark}
It is worth noting that the smallness condition on $d$ is only a convenient
sufficient condition for ensuring the invertibility of $W_{\mathrm{o}}+L(d)$.
Indeed, for a given system, any $d>0$ satisfying $\det(W_{\mathrm{o}%
}+L(d))\neq0$ is admissible. Thus, our method does not essentially require $d$
to be sufficiently small; it only requires $d$ to avoid the singular values at
which $W_{\mathrm{o}}+L(d)$ becomes singular. The same convention will be used
in the sequel: whenever a smallness condition on $d$ is introduced to
guarantee invertibility, it should be understood as a convenient sufficient
condition, while any $d$ preserving the required nonsingularity is also admissible.
\end{remark}
By construction, $\bar K(t)$ and $K(t)$ are $3h$-periodic.
Hence, the quantities defined below are independent of
$l\in\mathbb Z_0^+$.
Denote
\[
\alpha=\max\{\alpha_{1},\alpha_{2}\},
\]
in which
\begin{align}
\alpha_1
&= \begin{array}{ll}
	\mathrm{e}^{3|A|h} \int_{2h}^{3h} |\mathrm{e}^{-Av}| \, |B|
	\sum_{i=0}^{r-1} |\bar{K}_{i}(v+3lh)|
\end{array} \nonumber \\
&\quad \begin{array}{ll}
	\times |(CA^{i}+L_{i}(d))\mathrm{e}^{A(v-h)}| \, \mathrm{d}v
	\, (\delta_{1\infty}+\delta_{2\infty})
\end{array} \nonumber \\
&\quad + \begin{array}{ll}
	(\mathrm{e}^{3|A|h}+1)(\delta_{1\infty}+\delta_{2\infty}),
\end{array} \label{alpha_1} \\
\alpha_2
&= \begin{array}{ll}
	\mathrm{e}^{3|A|h}(\delta_{1\infty}+\delta_{2\infty})
	+ \mathrm{e}^{3|A|h} \int_{0}^{3h} \mathrm{e}^{|A|s} |B| \, \mathrm{d}s,
\end{array} \label{alpha_2}
\end{align}
with%
%\begin{align}
%\delta_{1\infty}  &  =\mathrm{e}^{3\left\vert A\right\vert h}\int_{0}%
%^{2h}\mathrm{e}^{\left\vert A\right\vert s}\left\vert B\right\vert
%\mathrm{d}s+\mathrm{e}^{3\left\vert A\right\vert h}\int_{2h}^{3h}%
%\mathrm{e}^{\left\vert A\right\vert v}\nonumber\\
%&  \times\left\vert B\right\vert
%%TCIMACRO{\dsum _{i=0}^{r-1}}%
%%BeginExpansion
%{\displaystyle\sum_{i=0}^{r-1}}
%%EndExpansion
%\left\vert \bar{K}_{i}(v+3lh)\right\vert \left\vert CA^{i}\right\vert
%\nonumber\\
%&  \times\int_{0}^{v-h}\left\vert \mathrm{e}^{A\left(  v-h-s\right)
%}B\right\vert \mathrm{d}s\mathrm{d}v+\mathrm{e}^{3\left\vert A\right\vert
%h}\int_{2h}^{3h}\mathrm{e}^{\left\vert A\right\vert v}\left\vert B\right\vert
%\nonumber\\
%&  \times%
%%TCIMACRO{\dsum _{i=1}^{r-1}}%
%%BeginExpansion
%{\displaystyle\sum_{i=1}^{r-1}}
%%EndExpansion
%\left\vert \bar{K}_{i}(v+3lh)\right\vert \int_{-id}^{0}\left\vert \varphi
%_{i}\left(  -s\right)  \right\vert \left\vert CA^{i+1}\right\vert \nonumber\\
%&  \times\int_{0}^{s+v-h}\left\vert \mathrm{e}^{A\left(  s+v-h-u\right)
%}B\right\vert \mathrm{d}u\mathrm{d}s\mathrm{d}v\nonumber\\
%&  +\mathrm{e}^{3\left\vert A\right\vert h}\int_{2h}^{3h}\mathrm{e}%
%^{\left\vert A\right\vert v}\left\vert B\right\vert \mathrm{d}v,
%\label{delta_1}%
%\end{align}
\begin{align}
\delta_{1\infty}
&= \begin{array}{ll}
	\mathrm{e}^{3|A|h} \int_{0}^{2h} \mathrm{e}^{|A|s} |B| \, \mathrm{d}s
\end{array} \nonumber \\
&\quad + \begin{array}{ll}
	\mathrm{e}^{3|A|h} \int_{2h}^{3h} \mathrm{e}^{|A|v} |B|
	\sum_{i=0}^{r-1} |\bar{K}_{i}(v+3lh)| \, |CA^{i}|
\end{array} \nonumber \\
&\quad \begin{array}{ll}
	\times \int_{0}^{v-h} |\mathrm{e}^{A(v-h-s)}B| \, \mathrm{d}s \, \mathrm{d}v
\end{array} \nonumber \\
&\quad + \begin{array}{ll}
	\mathrm{e}^{3|A|h} \int_{2h}^{3h} \mathrm{e}^{|A|v} |B|
	\sum_{i=1}^{r-1} |\bar{K}_{i}(v+3lh)|
\end{array} \nonumber \\
&\quad \begin{array}{ll}
	\times \int_{-id}^{0} |\varphi_{i}(-s)| \, |CA^{i+1}|
\end{array} \nonumber \\
&\quad \begin{array}{ll}
	\times \int_{0}^{s+v-h} |\mathrm{e}^{A(s+v-h-u)}B| \, \mathrm{d}u \, \mathrm{d}s \, \mathrm{d}v
\end{array} \nonumber \\
&\quad + \begin{array}{ll}
	\mathrm{e}^{3|A|h} \int_{2h}^{3h} \mathrm{e}^{|A|v} |B|
	|\bar{K}_{r-1}(v+3lh)|
\end{array} \nonumber \\
&\quad \begin{array}{ll}
	\times \int_{-(r-1)d}^{0}|\varphi_{r-1}(-s)|
	|CA^{r-1}B|\,\mathrm{d}s\,\mathrm{d}v
\end{array} \nonumber \\
&\quad + \begin{array}{ll}
	\mathrm{e}^{3|A|h} \int_{2h}^{3h} \mathrm{e}^{|A|v} |B| \, \mathrm{d}v,
\end{array} \label{delta_1} \\
\delta_{2\infty}
&= \begin{array}{ll}
	\mathrm{e}^{3|A|h} \int_{2h}^{3h} \mathrm{e}^{|A|v} |B|
	\sum_{i=0}^{r-1} |K_{i}(v+3lh)| \, \mathrm{d}v.
\end{array} \label{delta_2}
\end{align}
%\begin{equation}
%\delta_{2\infty}=\mathrm{e}^{3\left\vert A\right\vert h}\int_{2h}%
%^{3h}\mathrm{e}^{\left\vert A\right\vert v}\left\vert B\right\vert
%%TCIMACRO{\dsum _{i=0}^{r-1}}%
%%BeginExpansion
%{\displaystyle\sum_{i=0}^{r-1}}
%%EndExpansion
%\left\vert K_{i}(v+3lh)\right\vert \mathrm{d}v. \label{delta_2}%
%\end{equation}
Here $K_{i}(t)\in\mathbb{R}^{m\times p}$ for $i=0,1,\ldots,r-1$ are the block
constituents of $K(t)$, i.e.,
\begin{equation}
K(t)=[K_{0}(t),K_{1}(t),\ldots,K_{r-1}(t)]\in\mathbf{%
%TCIMACRO{\U{211d} }%
%BeginExpansion
\mathbb{R}
%EndExpansion
}^{m\times n}, \label{K_block}%
\end{equation}
such that
\begin{equation}
K(t)=\bar{K}(t)M^{\mathrm{T}},
\label{K_Kbar_M_relation}
\end{equation}
with $\bar{K}(t)$ defined in (\ref{Kbar_block_def}) and $M$ being the block
upper triangular matrix whose $(i,j)$-th block is given in (\ref{Mij}).

We are in a position to state the following result:

\begin{theorem}
\label{the_2}Consider system (\ref{sys_2}) under Assumptions \ref{ass1}%
--\ref{ass2}. Given $h\geq(r-1)d$, choose $d>0$ such that $W_{\mathrm{o}%
}+L(d)$ is invertible. Consider the smooth time-varying static output feedback
controller
\begin{equation}
u\left(  t\right)  =%
%TCIMACRO{\dsum _{i=0}^{r-1}}%
%BeginExpansion
{\displaystyle\sum_{i=0}^{r-1}}
%EndExpansion
g(t)K_{i}(t)y\left(  t-id-h\right)  , \label{controller}%
\end{equation}
where $K_{i}(t)\in\mathbb{R}^{m\times p}$, $i=0,1,\ldots,r-1$, are the block
constituents of $K(t)\in\mathbb{R}^{m\times n}$, as in (\ref{K_block}),
satisfying $K(t)=\bar K(t)M^{\mathrm{T}}$, with $M$ and $\bar K(t)$ being
defined in (\ref{Mij}) and (\ref{K_hat}), respectively. Then the solution to
the closed-loop system satisfies
\begin{equation}
\left\vert x(t)\right\vert \leq\alpha\cdot\left(  w_{1\infty}+w_{2\infty
}\right)  ,\quad t\geq3h, \label{x_iss}%
\end{equation}
meaning, prescribed-time ISS within a prescribed time $3h$. In particular, if $w_{1\infty}=0$ and
$w_{2\infty}=0$, we have $x(t)=0$, for $t\geq3h$, meaning,
prescribed-time stable within a prescribed time $3h$.
\end{theorem}

\begin{proof}
For $t\in\lbrack3lh,2h+3lh]$, $l\in%
%TCIMACRO{\U{2124} }%
%BeginExpansion
\mathbb{Z}
%EndExpansion
_{0}^{+}$, in view of$\ g\left(  t\right)  =0$, we have $\dot{x}\left(
t\right)  =Ax\left(  t\right)  +Bw_{1}\left(  t\right)  $, whose solution can
be expressed as
\begin{equation}
x\left(  t\right)  =\mathrm{e}^{A\left(  t-3lh\right)  }x\left(  3lh\right)
+\int_{3lh}^{t}\mathrm{e}^{A\left(  t-s\right)  }Bw_{1}\left(  s\right)
\mathrm{d}s. \label{x_t}%
\end{equation}
For $t\in\lbrack3lh+2h,3\left(  l+1\right)  h]$, $l\in%
%TCIMACRO{\U{2124} }%
%BeginExpansion
\mathbb{Z}
%EndExpansion
_{0}^{+}$, by noting $g\left(  t\right)  =1$, and using
(\ref{y_norm_def}), (\ref{y_i_hat}), (\ref{y_h}), and
(\ref{K_Kbar_M_relation}), the closed-loop system (\ref{sys_2}) and
(\ref{controller}) can be represented as
\begin{align}
\dot{x}(t) =  &  Ax(t) + B \sum_{i=0}^{r-1} K_{i}(t) y_{0}(t-id-h)\nonumber\\
&  + B \sum_{i=0}^{r-1} K_{i}(t) w_{2}(t-id-h) + B w_{1}(t)\nonumber\\
=  &  Ax(t) + B \sum_{i=0}^{r-1} \bar{K}_{i}(t) \tilde{y}_{i}(t-h)\nonumber\\
&  + B \sum_{i=0}^{r-1} K_{i}(t) w_{2}(t-id-h) + B w_{1}(t)\nonumber\\
=  &  Ax(t) + B \sum_{i=0}^{r-1} \bar{K}_{i}(t) C A^{i} x(t-h)
+B w_{1}(t)\nonumber\\
&  + B \sum_{i=0}^{r-1} R_{i}(t)
+B \sum_{i=0}^{r-1} K_{i}(t) w_{2}(t-id-h),\label{x_dot}%
\end{align}
where for $i=0,1,\ldots,r-1$,%
\[
R_{i}\left(  t\right)
=-\bar{K}_{i}(t)\int_{t-id-h}^{t-h}\varphi_{i}\left(
t-h-s\right)y_{0}^{\left(i+1\right)}\left(s\right)\mathrm{d}s.
\]
Since $t-h\in\lbrack3lh,3lh+2h]$ and
$t-id-h\in\lbrack3lh,3lh+2h]$, we have $u(s)=0$ on the corresponding
integration intervals. By using (\ref{cab}), (\ref{y_norm_def}), and (\ref{x_t}), we obtain
\[
y_{0}^{(i+1)}(s)
=
CA^{i+1}x(s)
+
\eta_i CA^{r-1}Bw_{1}(s),
\]
where $\eta_i=1$ if $i=r-1$ and
$\eta_i=0$ otherwise. Therefore,
\begin{align}
&  \int_{t-id-h}^{t-h}\varphi_{i}(t-h-s)y_{0}^{(i+1)}(s)\mathrm{d}%
s\nonumber\\
&  =\int_{t-id-h}^{t-h}\varphi_{i}(t-h-s)CA^{i+1}x(s)\mathrm{d}s\nonumber\\
&\quad+\eta_i
\int_{t-id-h}^{t-h}\varphi_{i}(t-h-s)
CA^{r-1}Bw_{1}(s)\mathrm{d}s\nonumber\\
&  =\int_{t-id-h}^{t-h}\varphi_{i}(t-h-s)CA^{i+1}
\mathrm{e}^{A(s-3lh)}\mathrm{d}s x(3lh)\nonumber\\
&\quad+\int_{t-id-h}^{t-h}\varphi_{i}(t-h-s)CA^{i+1}
\int_{3lh}^{s}\mathrm{e}^{A(s-v)}Bw_{1}(v)
\mathrm{d}v\mathrm{d}s\nonumber\\
&\quad+\eta_i
\int_{t-id-h}^{t-h}\varphi_{i}(t-h-s)
CA^{r-1}Bw_{1}(s)\mathrm{d}s. \label{11}%
\end{align}
Substituting (\ref{11}) into (\ref{x_dot}) yields
\begin{align}
\dot{x}( t) =  &  Ax( t) +B%
%TCIMACRO{\dsum _{i=0}^{r-1}}%
%BeginExpansion
{\displaystyle\sum_{i=0}^{r-1}}
%EndExpansion
\bar{K}_{i}(t)CA^{i}\mathrm{e}^{A( t-h-3lh) }x( 3lh)\nonumber\\
&  -B%
%TCIMACRO{\dsum _{i=0}^{r-1}}%
%BeginExpansion
{\displaystyle\sum_{i=0}^{r-1}}
%EndExpansion
\bar{K}_{i}(t)\int_{t-id-h}^{t-h}\varphi_{i}( t-h-s)\nonumber\\
&  \times CA^{i+1}\mathrm{e}^{A( s-3lh) }\mathrm{d}s x( 3lh)
+\psi(t),\label{x_dot3}%
\end{align}
where
\begin{align}
\psi(t) =  &  B%
%TCIMACRO{\dsum _{i=0}^{r-1}}%
%BeginExpansion
{\displaystyle\sum_{i=0}^{r-1}}
%EndExpansion
\bar{K}_{i}(t)CA^{i}\int_{3lh}^{t-h}
\mathrm{e}^{A(t-h-s)}Bw_{1}(s)\mathrm{d}s\nonumber\\
&  -B%
%TCIMACRO{\dsum _{i=0}^{r-1}}%
%BeginExpansion
{\displaystyle\sum_{i=0}^{r-1}}
%EndExpansion
\bar{K}_{i}(t)\int_{t-id-h}^{t-h}
\varphi_{i}(t-h-s)CA^{i+1}\nonumber\\
&  \times\int_{3lh}^{s}\mathrm{e}^{A(s-v)}Bw_{1}(v)
\mathrm{d}v\mathrm{d}s\nonumber\\
&  -B\bar{K}_{r-1}(t)
\int_{t-(r-1)d-h}^{t-h}\varphi_{r-1}(t-h-s)\nonumber\\
&  \times CA^{r-1}Bw_{1}(s)\mathrm{d}s\nonumber\\
&  +B%
%TCIMACRO{\dsum _{i=0}^{r-1}}%
%BeginExpansion
{\displaystyle\sum_{i=0}^{r-1}}
%EndExpansion
K_{i}(t)w_{2}(t-id-h)+Bw_{1}(t) . \label{phi_1}%
\end{align}
By the variation of constants formula, the solution $x(t)$ to
(\ref{x_dot3}) can be expressed as
\begin{equation}
x(t)=\mathit{\Delta}(t)x(3lh)+\pi(t), \label{x_t2}%
\end{equation}
where%
\begin{align*}
\mathit{\Delta}(t)
&=\mathrm{e}^{A(t-3lh)}\\
&\quad+\int_{3lh+2h}^{t}\mathrm{e}^{A(t-v)}B%
%TCIMACRO{\dsum _{i=0}^{r-1}}%
%BeginExpansion
{\displaystyle\sum_{i=0}^{r-1}}
%EndExpansion
\bar{K}_{i}(v)E_{i}(d)\mathrm{e}^{A(v-h)}
\mathrm{d}v\mathrm{e}^{-3lhA},
\end{align*}
and
\begin{align*}
\pi(t)=  &  \mathrm{e}^{A(t-3lh)}
\int_{3lh}^{3lh+2h}\mathrm{e}^{A(3lh-s)}
Bw_{1}(s)\mathrm{d}s\\
&  +\int_{3lh+2h}^{t}\mathrm{e}^{A(t-v)}
\psi(v)\mathrm{d}v,
\end{align*}
with $E_{i}(d)$ for $i=0,1,\ldots,r-1$ being the block constituents of
$E(d)$, i.e.,
\begin{equation}
E(d)=[E_{0}^{\mathrm{T}}(d),E_{1}^{\mathrm{T}}(d),\ldots,
E_{r-1}^{\mathrm{T}}(d)]^{\mathrm{T}}%
=W_{\mathrm{o}}+L(d), \label{E_def}%
\end{equation}
where $L(d)$ is defined in (\ref{L_h_whole_def}).
Therefore, for $t=3(l+1)h$, from (\ref{K_hat}) it follows that
\begin{align*}
\mathit{\Delta}(3(l+1)h)
=  &  \mathrm{e}^{3Ah}\left(
I_{n}+\int_{2h}^{3h}\mathrm{e}^{-Av}B\bar{K}(v+3lh)\right.\\
&  \times\left.
(W_{\mathrm{o}}+L(d))\mathrm{e}^{A(v-h)}\mathrm{d}v\right)
=0_{n}.
\end{align*}
Thus we have $x(3(l+1)h)=\pi(3(l+1)h)$. In view of
(\ref{delta_1}) and (\ref{delta_2}), $\pi(t)$ satisfies
$\vert\pi(t)\vert\leq\delta_{1\infty}w_{1\infty}
+\delta_{2\infty}w_{2\infty}$. For
$t\in\lbrack3lh,3(l+1)h]$, $l\in%
%TCIMACRO{\U{2115} }%
%BeginExpansion
\mathbb{N}
%EndExpansion
$, in view of (\ref{x_t}) and (\ref{x_t2}), we obtain
(\ref{x_iss}).
\end{proof}

\begin{remark}
\label{rem1}According to (\ref{Mij}), we have
$\lim_{d\to0^{+}}\vert M\vert=\infty$. From
(\ref{alpha_1}) and (\ref{alpha_2}), it can be observed that a very small
$d$ may cause $\alpha$ in (\ref{x_iss}) to become large, meaning that the disturbances are amplified.
\end{remark}

\subsection{Systems in Controllable Canonical Form}

A more explicit controller gain will be presented in this section for systems
in canonical form. Consider (\ref{sys_2}) with $x(t)=[x_{0}^{\mathrm{T}%
}(t),x_{1}^{\mathrm{T}}(t),\ldots,x_{r-1}^{\mathrm{T}}(t)]^{\mathrm{T}}%
\in\mathbf{%
%TCIMACRO{\U{211d} }%
%BeginExpansion
\mathbb{R}
%EndExpansion
}^{rp}$, $x_{i}(t)\in\mathbf{%
%TCIMACRO{\U{211d} }%
%BeginExpansion
\mathbb{R}
%EndExpansion
}^{p}$, $A_{i}\in\mathbf{%
%TCIMACRO{\U{211d} }%
%BeginExpansion
\mathbb{R}
%EndExpansion
}^{p\times p}$, $\bar{B}\in\mathbf{%
%TCIMACRO{\U{211d} }%
%BeginExpansion
\mathbb{R}
%EndExpansion
}^{p\times m}$, $u\left(  t\right)  \in\mathbf{%
%TCIMACRO{\U{211d} }%
%BeginExpansion
\mathbb{R}
%EndExpansion
}^{m}$,
\begin{align}
A  &  =\left[
\begin{array}
[c]{ccccc}%
0 & I_{p} & 0 & \cdots & 0\\
0 & 0 & I_{p} & \cdots & 0\\
\vdots & \vdots & \vdots & \vdots & \vdots\\
0 & 0 & 0 & \cdots & I_{p}\\
A_{0} & A_{1} & A_{2} & \cdots & A_{r-1}%
\end{array}
\right]  \in\mathbf{%
%TCIMACRO{\U{211d} }%
%BeginExpansion
\mathbb{R}
%EndExpansion
}^{rp\times rp},\label{A}\\
B  &  =\left[
\begin{array}
[c]{ccccc}%
0^{\mathrm{T}} & 0^{\mathrm{T}} & \cdots & 0^{\mathrm{T}} & \bar
{B}^{\mathrm{T}}%
\end{array}
\right]  ^{\mathrm{T}}\in\mathbf{%
%TCIMACRO{\U{211d} }%
%BeginExpansion
\mathbb{R}
%EndExpansion
}^{rp\times m},\label{B}\\
C  &  =\left[
\begin{array}
[c]{ccccc}%
I_{p} & 0 & \cdots & 0 & 0
\end{array}
\right]  \in\mathbf{%
%TCIMACRO{\U{211d} }%
%BeginExpansion
\mathbb{R}
%EndExpansion
}^{p\times rp}. \label{C}%
\end{align}

Denote
\begin{equation}
\mathcal{L}(d)=\big[\,\mathcal{L}_{0}^{\mathrm{T}}(d),\
\mathcal{L}_{1}^{\mathrm{T}}(d),\ \ldots,\
\mathcal{L}_{r-1}^{\mathrm{T}}(d)\,\big]^{\mathrm{T}}
\in\mathbb{R}^{rp\times rp},
\label{calL_def}
\end{equation}
where, for $i=0,1,\ldots,r-1$, each block row $\mathcal{L}_{i}(d)\in
\mathbb{R}^{p\times rp}$ is given by
\begin{equation}
\mathcal{L}_{i}(d)=\frac{(-1)^{r}}{(r-1)!}\int_{-d_{i}}^{0}(v+d_{i}%
)^{r-1}e_{p,r-1}^{\mathrm{T}}A\mathrm{e}^{Av}\mathrm{d}v,\label{L_i}%
\end{equation}
with $d_{i}=id$ and $e_{p,r-1}=[0,\ldots,0,I_{p}]^{\mathrm{T}}\in
\mathbb{R}^{rp\times p}$. The matrix $\mathcal{M}\in\mathbb{R}^{rp\times rp}$
is defined as
\begin{equation}
\mathcal{M}=\left[
\begin{array}
[c]{cccc}%
I_{p} & 0 & \cdots & 0\\
I_{p} & -d_{1}I_{p} & \cdots & \frac{(-d_{1})^{r-1}}{(r-1)!}I_{p}\\
\vdots & \vdots & \vdots & \vdots\\
I_{p} & -d_{r-1}I_{p} & \cdots & \frac{(-d_{r-1})^{r-1}}{(r-1)!}I_{p}%
\end{array}
\right]  \in\mathbf{%
%TCIMACRO{\U{211d} }%
%BeginExpansion
\mathbb{R}
%EndExpansion
}^{rp\times rp}.\label{Mm}%
\end{equation}

Define
\begin{align}
\bar{\mathcal{K}}(t) =  &  -S(t)B^{\mathrm{T}}\mathrm{e}^{A^{\mathrm{T}%
}(3lh-t)} W_{\mathrm{c}}\mathrm{e}^{-A(t-3lh-h)}\nonumber\\
&  \times\left(  I_{rp}+\mathcal{M}^{-1}\mathcal{L}(d) \right) ^{-1}%
,\label{K_t_hat}%
\end{align}
where $W_{\mathrm{c}}$ is defined in (\ref{W_c_def}). We next show that
$I_{rp}+\mathcal{M}^{-1}\mathcal{L}(d)$ is invertible for sufficiently small
$d$ and that the controller based on the gain $\bar{\mathcal{K}}(t)$
stabilizes the system in prescribed time.

Since $d_{i}=id$ for $i=0,1,\ldots,r-1$, Lemma 1 in
\cite{emilia17auto} gives $\mathcal{M}^{-1}=O(d^{-r+1})$. Moreover,
\begin{align*}
& \mathcal{L}_{i}(d)=  \frac{(-1)^{r}}{(r-1)!}\int_{-id}^{0}(u+id)^{r-1}%
e_{p,r-1}^{\mathrm{T}}A\mathrm{e}^{Au}\mathrm{d}u\\
\overset{u=dv}{=} &  \frac{(-1)^{r}}{(r-1)!}d^{r}\int_{-i}^{0}(v+i)^{r-1}%
e_{p,r-1}^{\mathrm{T}}A\mathrm{e}^{Adv}\mathrm{d}v= O(d^{r}),
\end{align*}
which implies that $\lim_{d\rightarrow0^{+}}\mathcal{M}^{-1}\mathcal{L}%
(d)=\lim_{d\rightarrow0^{+}}O(d)=0$. Therefore, there exists a sufficiently
small constant $d_{\ast}$, such that for $d\in(0,d_{\ast})$, $I_{rp}%
+\mathcal{M}^{-1}\mathcal{L}(d)$ is invertible. Consequently, $\bar
{\mathcal{K}}(t)$ is well defined. We next introduce several constants that
will be used to derive an a priori estimate for the closed-loop solutions.

To state the following estimate compactly, define the constants
\begin{equation}
\beta_{1}=\left(  \mathrm{e}^{3|A|h}+1+\mathrm{e}^{3|A|h}\lambda_{0}\right)
\lambda_{1}+\mathrm{e}^{3|A|h}\int_{0}^{3h}\mathrm{e}^{|A|s}|B|\mathrm{d}s,
\label{beta_1}%
\end{equation}
and
\begin{equation}
\beta_{2}=\mathrm{e}^{3|A|h}\lambda_{0}\lambda_{2}+\lambda_{2}+\mathrm{e}%
^{3|A|h}\lambda_{2}. \label{beta_2}%
\end{equation}
Here, the auxiliary constants $\lambda_{0}$, $\lambda_{1}$, and $\lambda_{2}$
are given by
\begin{align}
\lambda_0
&= \begin{array}{ll}
	1 + \int_{2h}^{3h} \mathrm{e}^{|A|s} |B| \, |\bar{\mathcal{K}}(s+3lh)|
\end{array} \nonumber \\
&\quad \begin{array}{ll}
	\times |(I_{rp}+\mathcal{M}^{-1}\mathcal{L}(d))| \, |\mathrm{e}^{A(s-h)}| \, \mathrm{d}s,
\end{array} \nonumber \\
\lambda_1
&= \begin{array}{ll}
	\mathrm{e}^{3|A|h} \int_{0}^{2h} \mathrm{e}^{|A|s} |B| \, \mathrm{d}s
	+ \mathrm{e}^{3|A|h} \int_{2h}^{3h} \mathrm{e}^{|A|u} |B|
\end{array} \nonumber \\
&\quad \begin{array}{ll}
	\times |\bar{\mathcal{K}}(u+3lh)| \left( \mathrm{e}^{3|A|h} \int_{0}^{u-h} |\mathrm{e}^{A(u-s-h)}| \, |B| \, \mathrm{d}s \right) \mathrm{d}u
\end{array} \nonumber \\
&\quad + \begin{array}{ll}
	\frac{\mathrm{e}^{3|A|h}}{(r-1)!} \int_{2h}^{3h} \mathrm{e}^{|A|u}
	\sum_{i=0}^{r-1} |B\mathcal{K}_{i}(u+3lh)|
\end{array} \nonumber \\
&\quad \begin{array}{ll}
	\times \int_{-d_{i}}^{0} |(s+d_{i})^{r-1} e_{p,r-1}^{\mathrm{T}} A|
	\int_{0}^{s+u-h} |\mathrm{e}^{A(s+u-h-v)}|
\end{array} \nonumber \\
&\quad \begin{array}{ll}
	\times |B| \, \mathrm{d}v \, \mathrm{d}s \, \mathrm{d}u+\mathrm{e}^{3|A|h} \int_{2h}^{3h} \mathrm{e}^{|A|s} |B| \, \mathrm{d}s
\end{array} \nonumber \\
&\quad + \begin{array}{ll}
	\frac{\mathrm{e}^{3|A|h}}{(r-1)!} \int_{2h}^{3h} \mathrm{e}^{|A|u}
	\sum_{i=0}^{r-1} |B\mathcal{K}_{i}(u+3lh)|
\end{array} \nonumber \\
&\quad \begin{array}{ll}
	\times \int_{-d_{i}}^{0} |(s+d_{i})^{r-1} e_{p,r-1}^{\mathrm{T}} B| \, \mathrm{d}s \, \mathrm{d}u.
\end{array} \label{lemma_1}
\end{align}
Moreover,
\begin{align}
&\lambda_2
= \begin{array}{ll}
	\mathrm{e}^{3|A|h} \int_{2h}^{3h} \mathrm{e}^{|A|u}
	\sum_{i=0}^{r-1} |B\mathcal{K}_{i}(u+3lh)| \, \mathrm{d}u.
\end{array} \label{lemma_2}
\end{align}
In the above expressions, $\mathcal{K}_{i}(t)\in\mathbb{R}^{m\times p}$,
$i=0,1,\ldots,r-1$, are the block columns of $\mathcal{K}(t)$, namely,
\begin{equation}
\mathcal{K}(t)=[\mathcal{K}_{0}(t),\mathcal{K}_{1}(t),\ldots,\mathcal{K}%
_{r-1}(t)]\in\mathbb{R}^{m\times rp}, \label{Ka_block}%
\end{equation}
where $\mathcal{K}(t)=\bar{\mathcal{K}}(t)\mathcal{M}^{-1}$ with $\bar{\mathcal{K}}(t)$ defined in (\ref{K_t_hat}) and $\mathcal{M}$
defined in (\ref{Mm}).

With the above preparation, we obtain the following conclusion.

\begin{theorem}
\label{the3}Consider system (\ref{sys_2}) with $(A,B,C)$ defined in
(\ref{A})--(\ref{C}) under Assumption \ref{ass1}. Given
$h\geq (r-1)d$, choose $d>0$ such that $I_{rp}+\mathcal{M}^{-1}\mathcal{L}\left(
d\right)  $ is invertible. Consider the smooth time-varying static output
feedback controller
\begin{equation}
u(t)=\sum_{i=0}^{r-1} g(t)\mathcal{K}_{i}(t)\, y(t-d_{i}-h),\label{output}%
\end{equation}
where for $i=0,1,\ldots,r-1$, $\mathcal{K}_{i}(t)\in\mathbb{R}^{m\times p}$ are
the block columns of $\mathcal{K}(t)\in\mathbb{R}^{m\times rp}$, as in
(\ref{Ka_block}), satisfying $\mathcal{K}(t)=\mathcal{\bar K}(t)\mathcal{M}^{-1}$
with $\mathcal{M}$ and $\mathcal{\bar K}(t)$ defined in (\ref{Mm}) and
(\ref{K_t_hat}), respectively. Then the solution to the closed-loop system
(\ref{sys_2}) and (\ref{output}) satisfies (\ref{x_iss}) with $\alpha=\max\{\beta_{1},\beta_{2}\}$, meaning, prescribed-time ISS within a prescribed time $3h$.
\end{theorem}

\begin{proof}
In view of (\ref{sys_2}), by Taylor expansion with the integral
form, we have for $i=0,1,\ldots,r-1$,
\begin{align}
x_{0}(t-d_{i}) =  &
%TCIMACRO{\dsum _{j=0}^{r-1}}%
%BeginExpansion
{\displaystyle\sum_{j=0}^{r-1}}
%EndExpansion
\frac{1}{j!}(-d_{i})^{j}\frac{\mathrm{d}^{j}x_{0}(t)}{\mathrm{d}t^{j}%
}\nonumber\\
&  +\frac{(-1)^{r}}{(r-1)!}\int_{t-d_{i}}^{t}(s-t+d_{i})^{r-1}\frac
{\mathrm{d}^{r}x_{0}(s)}{\mathrm{d}s^{r}}\mathrm{d}s\nonumber\\
=  &
%TCIMACRO{\dsum _{j=0}^{r-1}}%
%BeginExpansion
{\displaystyle\sum_{j=0}^{r-1}}
%EndExpansion
\frac{1}{j!}(-d_{i})^{j}x_{j}(t)+\mathcal{R}_{i}(t), \label{x_0}%
\end{align}
where
\begin{equation}
\mathcal{R}_{i}(t)=\frac{(-1)^{r}}{(r-1)!}\int_{t-d_{i}}^{t}(s-t+d_{i}%
)^{r-1}x_{r}(s)\mathrm{d}s,\label{R_i01}%
\end{equation}
with $x_{r}(s)=\dot{x}_{r-1}(s)$.

Then the closed-loop system consisting of (\ref{sys_2}) and (\ref{output}) can
be written as
\begin{align}
\dot{x}(t) =  &  Ax(t)+%
%TCIMACRO{\dsum _{i=0}^{r-1}}%
%BeginExpansion
{\displaystyle\sum_{i=0}^{r-1}}
%EndExpansion
Bg(t)\mathcal{K}_{i}(t)x_{0}(t-d_{i}-h)\nonumber\\
&  +%
%TCIMACRO{\dsum _{i=0}^{r-1}}%
%BeginExpansion
{\displaystyle\sum_{i=0}^{r-1}}
%EndExpansion
Bg(t)\mathcal{K}_{i}(t)w_{2}(t-d_{i}-h) +Bw_{1}(t). \label{x_dot_clo}%
\end{align}
According to (\ref{x_0}) and  $\mathcal{K}(t)=\mathcal{\bar K}(t)\mathcal{M}^{-1}$, we obtain
\begin{align*}
&
%TCIMACRO{\dsum _{i=0}^{r-1}}%
%BeginExpansion
{\displaystyle\sum_{i=0}^{r-1}}
%EndExpansion
\mathcal{K}_{i}(t)x_{0}(t-d_{i}-h)\\
&  =%
%TCIMACRO{\dsum _{j=0}^{r-1}}%
%BeginExpansion
{\displaystyle\sum_{j=0}^{r-1}}
%EndExpansion
\mathcal{\bar{K}}_{j}(t)x_{j}(t-h)+%
%TCIMACRO{\dsum _{i=0}^{r-1}}%
%BeginExpansion
{\displaystyle\sum_{i=0}^{r-1}}
%EndExpansion
\mathcal{K}_{i}(t)\mathcal{R}_{i}(t-h).
\end{align*}
Substituting this into (\ref{x_dot_clo}) yields
\begin{align}
 &\dot{x}(t)  =Ax(t)+Bg(t)\mathcal{\bar{K}}(t)x(t-h)+Bw_{1}(t)\nonumber\\
&  +%
%TCIMACRO{\dsum _{i=0}^{r-1}}%
%BeginExpansion
{\displaystyle\sum_{i=0}^{r-1}}
%EndExpansion
Bg(t)\mathcal{K}_{i}(t)(\mathcal{R}_{i}(t-h)+w_{2}(t-d_{i}-h))
.\label{x_dot2}%
\end{align}

For $t\in\lbrack3lh,2h+3lh]$, $l\in%
%TCIMACRO{\U{2124} }%
%BeginExpansion
\mathbb{Z}
%EndExpansion
_{0}^{+}$, in view of $g(t)=0$, (\ref{x_dot2}) reduces to $\dot{x}%
(t)=Ax(t)+Bw_{1}(t)$, whose solution can be calculated as
\begin{equation}
x(t)=\mathrm{e}^{A(t-3lh)}x(3lh)+\int_{3lh}^{t}\mathrm{e}^{A(t-s)}%
Bw_{1}(s)\mathrm{d}s.\label{x_t_part1}%
\end{equation}
For $t\in\lbrack2h+3lh,3(l+1)h]$, $l\in%
%TCIMACRO{\U{2124} }%
%BeginExpansion
\mathbb{Z}
%EndExpansion
_{0}^{+}$, by noting $g(t)=1$, (\ref{x_dot2}) can be rewritten as
\begin{align}
\dot{x}(t)= &  Ax(t)+B\mathcal{\bar{K}}(t)x(t-h)+%
%TCIMACRO{\dsum _{i=0}^{r-1}}%
%BeginExpansion
{\displaystyle\sum_{i=0}^{r-1}}
%EndExpansion
B\mathcal{K}_{i}(t)\mathcal{R}_{i}(t-h)\nonumber\\
&  +%
%TCIMACRO{\dsum _{i=0}^{r-1}}%
%BeginExpansion
{\displaystyle\sum_{i=0}^{r-1}}
%EndExpansion
B\mathcal{K}_{i}(t)w_{2}(t-d_{i}-h)+Bw_{1}(t).\label{x_t_cl2}%
\end{align}
Since $t-h\in\lbrack3lh,2h+3lh]$ and $t-d_{i}-h\in\lbrack3lh,2h+3lh]$ for
$d_{i}\leq h$, $x_{r}(s)$ in $\mathcal{R}_{i}(t-h)$ can be computed as
\begin{align*}
x_{r}(t)= &  \dot{x}_{r-1}(t)=e_{p,r-1}^{\mathrm{T}}\dot{x}(t)=e_{p,r-1}%
^{\mathrm{T}}(Ax(t)+Bw_{1}(t))\\
= &  e_{p,r-1}^{\mathrm{T}}A\mathrm{e}^{A(t-3lh)}x(3lh)\\
&  +e_{p,r-1}^{\mathrm{T}}A\int_{3lh}^{t}\mathrm{e}^{A(t-s)}Bw_{1}%
(s)\mathrm{d}s+e_{p,r-1}^{\mathrm{T}}Bw_{1}(t).
\end{align*}
With this and (\ref{x_t_part1}), (\ref{x_t_cl2}) can be continued as
\begin{equation}
\dot{x}(t)=Ax(t)+B\mathit{\Theta}(t)\mathrm{e}^{-3lhA}x(3lh)+\mathit{\Lambda
}(t),\label{x_23}%
\end{equation}
where
\begin{equation}
\mathit{\Theta}(t)=\mathcal{\bar{K}}(t)\mathrm{e}^{A(t-h)}+\mathcal{K}%
(t)\mathcal{L}(d)\mathrm{e}^{A(t-h)},\label{cetatt}%
\end{equation}
and%
\begin{align*}
\mathit{\Lambda}(t)= &  B\mathcal{\bar{K}}(t)\int_{3lh}^{t-h}\mathrm{e}%
^{A(t-s-h)}Bw_{1}(s)\mathrm{d}s\\
&  +%
%TCIMACRO{\dsum _{i=0}^{r-1}}%
%BeginExpansion
{\displaystyle\sum_{i=0}^{r-1}}
%EndExpansion
B\mathcal{K}_{i}(t)\frac{(-1)^{r}}{(r-1)!}\int_{t-d_{i}-h}^{t-h}%
(s-t+h+d_{i})^{r-1}\\
&  \times e_{p,r-1}^{\mathrm{T}}A\int_{3lh}^{s}\mathrm{e}^{A(s-v)}%
Bw_{1}(v)\mathrm{d}v\mathrm{d}s\\
&  +%
%TCIMACRO{\dsum _{i=0}^{r-1}}%
%BeginExpansion
{\displaystyle\sum_{i=0}^{r-1}}
%EndExpansion
B\mathcal{K}_{i}(t)\frac{(-1)^{r}}{(r-1)!}\int_{t-d_{i}-h}^{t-h}%
(s-t+h+d_{i})^{r-1}\\
&  \times e_{p,r-1}^{\mathrm{T}}Bw_{1}(s)\mathrm{d}s\\
&  +%
%TCIMACRO{\dsum _{i=0}^{r-1}}%
%BeginExpansion
{\displaystyle\sum_{i=0}^{r-1}}
%EndExpansion
B\mathcal{K}_{i}(t)w_{2}(t-d_{i}-h)+Bw_{1}(t),
\end{align*}
with $\mathcal{L}(d)$ being defined in (\ref{calL_def}).
By the variation of constants formula, the solution to (\ref{x_23}) can be
calculated as
\begin{equation}
x(t)=\mathit{\Delta}(t)x(3lh)+\mathit{\Gamma}(t),\label{x_t_cl3}%
\end{equation}
where
\begin{align}
\mathit{\Delta}(t)= &  \mathrm{e}^{A(t-3lh)}+\int_{2h+3lh}^{t}\mathrm{e}%
^{A(t-s)}B\mathit{\Theta}(s)\mathrm{d}s\mathrm{e}^{-3lhA},\nonumber\\
\mathit{\Gamma}(t)= &  \int_{3lh}^{2h+3lh}\mathrm{e}^{A(t-s)}Bw_{1}%
(s)\mathrm{d}s\nonumber\\
&  +\int_{2h+3lh}^{t}\mathrm{e}^{A(t-s)}\mathit{\Lambda}(s)\mathrm{d}%
s.\label{kama}%
\end{align}
Then for $t=3(l+1)h$, we have $x(3(l+1)h)=\mathit{\Delta}(3(l+1)h)x(3lh)+\mathit{\Gamma}(3(l+1)h)$.
According to (\ref{K_t_hat}) and (\ref{cetatt}), we obtain $\mathit{\Delta
}(3(l+1)h)=0_{rp}.$ Thus $x(3(l+1)h)=\mathit{\Gamma}(3(l+1)h),l\in%
%TCIMACRO{\U{2124} }%
%BeginExpansion
\mathbb{Z}
%EndExpansion
_{0}^{+}$. In view of (\ref{lemma_1}) and (\ref{lemma_2}), (\ref{kama})
satisfies $|\mathit{\Gamma}(t)|\leq\lambda_{1}w_{1\infty}+\lambda
_{2}w_{2\infty}$. For $t\in\lbrack3lh,3(l+1)h],l\in\mathbb{N},$ in view of
(\ref{x_t_part1}) and (\ref{x_t_cl3}), we obtain (\ref{x_iss}) with $\alpha=\max\{\beta_{1},\beta_{2}\}$.
\end{proof}

Similar to Remark \ref{rem1}, in view of (\ref{Mm}) and Lemma 1 in
\cite{emilia17auto}, we have $|\mathcal{M}^{-1}|=O(d^{-r+1})$ as $d\to0^{+}$.
Thus, $|\mathcal{M}^{-1}|$ may become unbounded as $d\to0^{+}$. From
(\ref{beta_1}) and (\ref{beta_2}), it can be observed that a very small $d$
may amplify the effect of disturbances.

\section{\label{sec3}Fixed-Time Static Output Feedback in the Presence of Output Delay}
\subsection{Large Known Slowly-Varying Delay}
Consider the linear system
\begin{equation}
\left\{
\begin{array}
[c]{rl}%
\dot{x}\left(t\right) & =Ax\left(t\right)+Bu\left(t\right),\\
y\left(t\right) & =\left\{
\begin{array}
[c]{ll}%
Cx\left(t-\tau(t)\right), & t-\tau(t)\geq0,\\
0, & t-\tau(t)<0,
\end{array}
\right.
\end{array}
\right. \label{sys_vary}%
\end{equation}
where $\tau(t)\in[0,\tau_{M}]$ is a known slowly-varying delay satisfying $\dot{\tau}(t)\leq\mu<1$
with $\tau_{M}$ being known. For conceptual simplicity, we do not consider disturbances here, but our
results are extendable to additive disturbances similar to the previous
section.

Differently from the state-feedback case, large constant measurement delays cannot be compensated by linear static output-feedback controller with a constant gain by using variation of constants formula. Below, we briefly explain why a classical predictor cannot be used in the presence of a constant output delay. For system (\ref{sys_vary}) with $\tau(t)=\tau_{0}$, consider a control law as $u(t)=KCx(t-\tau_{0})$. Using a variation of constants formula, we have
\begin{align*}u(t)  =KC\mathrm{e}^{-A\tau_{0}}x(t)-KC\int_{t-\tau_{0}}^{t}\mathrm{e}^{A(t-\tau_{0}-s) }Bu(s)\mathrm{d}s.\end{align*}Since only $Cx(t)$ is available, to design a predictor we need $C\mathrm{e}^{-A\tau_{0}}=\mathrm{e}^{-A\tau_{0}}C$, which is equivalent to $CA=AC$. This conditions restrictive and never holds for $p<n$. Therefore, this classical method cannot compensate output delays. Surprisingly, our act and wait static output feedback compensates large output delay.

%\begin{remark}
%(Constant input delay) In the constant-delay case, input delays can be
%transferred to the measurement side by a time-axis transformation.
%Specifically, consider $\dot{x}(t)=Ax(t)+Bu(t-\tau_{1})$, where
%$\tau_{1}\leq\tau_{M1}$ is a known constant input delay. By defining
%$z(t)=x(t+\tau_{1})$, we obtain $\dot{z}(t)=Az(t)+Bu(t)$. Moreover, if the
%measured output is $y(t)=Cx(t-\tau(t))$, then
%$y(t)=Cz(t-\tau_{1}-\tau(t))$. Therefore, the constant input delay is
%converted into an additional measurement delay. In particular, if
%$\tau(t)\in[0,\tau_{M}]$, then the transformed measurement delay is bounded
%by $\tau_{M1}+\tau_{M}$.
%\end{remark}

Let $h>0$. Denote
\begin{equation}
\begin{aligned}
\mathcal{W}_{\mathrm{o}}(d,t)
=M^{\mathrm{T}}
\operatorname{col}_{j=0}^{r-1}
\left\{
C\mathrm{e}^{-A(jd+\tau(t-jd-h))}
\right\},
\end{aligned}
\label{W_o_hat}%
\end{equation}
where $M$ is defined by (\ref{Mij}).

Since $\tau(t)$ is known, the matrix
$\mathcal{W}_{\mathrm{o}}(d,t)$ can be calculated online. In the general
time-varying-delay case, we assume that there exists a constant $d>0$
such that
\begin{equation}
\inf_{t\geq2h}
\sigma_{\min}
\left(
\mathcal{W}_{\mathrm{o}}(d,t)
\right)>0.
\label{W_o_uniform}%
\end{equation}
That is, $\mathcal{W}_{\mathrm{o}}(d,t)$ is invertible for all $t\geq 2h$.

It is worth noting that condition (\ref{W_o_uniform}) is naturally
satisfied for sufficiently small $d$ when the measurement delay is
slowly varying. Assume that
\begin{equation}
\left\{
\begin{aligned}
&\tau\in\mathbf{C}^{r+1},\\
&\sup_{t\geq0}\left|\tau^{(i)}(t)\right|<\infty,
\quad i=1,\ldots,r+1,\\
&\dot{\tau}(t)\leq\mu<1,\qquad t\geq0.
\end{aligned}
\right.
\label{slow_delay}
\end{equation}

Observe that (\ref{y_i_hat}) is exactly the finite-difference formula.
In the present case,
$y_0(t)=Cx(t-\tau(t))$.
Let $q=t-h$ and define
$
F_q(s)=C\mathrm{e}^{A(s-\tau(q+s))}$.
It follows from (\ref{W_o_hat}) that the $i$-th block row of
$\mathcal W_{\mathrm o}(d,t)$ can be calculated as
\[
\begin{aligned}
\left[\mathcal W_{\mathrm o}(d,t)\right]_i
=
\frac{1}{d^i}
\sum_{j=0}^{i}
(-1)^j
\binom{i}{j}
F_q(-jd),
\end{aligned}
\]
which is precisely the counterpart of (\ref{y_i_hat}) for the delayed
output.

Since (\ref{slow_delay}) holds, the derivatives of $F_q$ involved
below are uniformly bounded with respect to $q\geq h$. Hence, the
finite differences converge uniformly to the corresponding derivatives
as $d\to0^{+}$, namely,
\[
\frac{1}{d^i}
\sum_{j=0}^{i}
(-1)^j
\binom{i}{j}
F_q(-jd)
\longrightarrow
F_q^{(i)}(0),
\]
uniformly with respect to $t\geq2h$.

By repeated application of the chain rule, we obtain
$
\operatorname{col}_{i=0}^{r-1}
\{F_{t-h}^{(i)}(0)\}
=
\mathcal T_\tau(t)
W_{\mathrm o}
\mathrm{e}^{-A\tau(t-h)}$,
where $\mathcal T_\tau(t)$ is a block lower triangular matrix whose
diagonal blocks are
$
I_p,
(1-\dot{\tau}(t-h))I_p,
\ldots,
(1-\dot{\tau}(t-h))^{r-1}I_p$.
Therefore, stacking all block rows together gives
\begin{equation}
\lim_{d\rightarrow0^+}
\sup_{t\geq2h}
\left|
\mathcal W_{\mathrm o}(d,t)
-
\mathcal T_\tau(t)
W_{\mathrm o}
\mathrm{e}^{-A\tau(t-h)}
\right|
=0.
\label{W_o_limit}
\end{equation}

Moreover, since $\mathcal T_\tau(t)$ is block lower triangular,
$\det\mathcal T_\tau(t)
=
(1-\dot{\tau}(t-h))^{pr(r-1)/2}$.
By (\ref{slow_delay}),
$1-\dot{\tau}(t)\geq1-\mu>0$, and all the entries of
$\mathcal T_\tau(t)$ are uniformly bounded. Hence, both
$\mathcal T_\tau(t)$ and $\mathcal T_\tau^{-1}(t)$ are uniformly
bounded. Moreover, $W_{\mathrm o}$ is nonsingular, and
$\mathrm{e}^{-A\tau(t-h)}$ and its inverse are uniformly bounded
because $\tau(t)\in[0,\tau_M]$. Thus, using
$\sigma_{\min}(X)=1/|X^{-1}|$ (see Section 7.3 in \cite{horn13}),
there exists a constant $\gamma>0$ such that
$
\inf_{t\geq2h}
\sigma_{\min}
\left(
\mathcal T_\tau(t)W_{\mathrm o}
\mathrm{e}^{-A\tau(t-h)}
\right)
\geq\gamma$.

By (\ref{W_o_limit}), there exists a sufficiently small constant
$d_\ast>0$ such that for all $d\in(0,d_\ast)$,
\[
\sup_{t\geq2h}
\left|
\mathcal W_{\mathrm o}(d,t)
-
\mathcal T_\tau(t)W_{\mathrm o}
\mathrm{e}^{-A\tau(t-h)}
\right|
<\frac{\gamma}{2}.
\]
Therefore, by the perturbation inequality of singular values (see Corollary 7.3.5 in \cite{horn13}),
$\sigma_{\min}(X)
\geq
\sigma_{\min}(Y)-|X-Y|$,
we obtain, for all $t\geq2h$,
$
\sigma_{\min}
\left(
\mathcal W_{\mathrm o}(d,t)
\right)
\geq\frac{\gamma}{2}>0$,
which proves that condition (\ref{W_o_uniform}) holds for all
$d\in(0,d_\ast)$.

With the above preparation, we can obtain the following conclusion.

\begin{theorem}
\label{the4}
Consider system (\ref{sys_vary}) under Assumptions
\ref{ass1}--\ref{ass2}. Given $h\geq(r-1)d+\tau_{M}$, choose $d>0$ such
that (\ref{W_o_uniform}) holds. Consider the time-varying static output
feedback controller
\begin{equation}
u(t)=\sum_{i=0}^{r-1}g(t)K_{i}(t)y(t-id-h), \label{u_vary}%
\end{equation}
where $K_{i}(t)\in\mathbb{R}^{m\times p}$ for
$i=0,1,\ldots,r-1$ are the block columns of
$K(t)\in\mathbb{R}^{m\times n}$, as defined in (\ref{K_block}), satisfying
$K(t)=\bar K(t)M^{\mathrm{T}}$ with $M\in\mathbb{R}^{n\times n}$ defined
in (\ref{Mij}) and $\bar K(t)$ given by
\begin{equation}
\bar{K}(t)
=-S(t)B^{\mathrm{T}}\mathrm{e}^{A^{\mathrm{T}}(3lh-t)}
W_{\mathrm{c}}\mathrm{e}^{-A(t-h-3lh)}
\mathcal{W}_{\mathrm{o}}^{-1}(d,t), \label{k123}%
\end{equation}
where $W_{\mathrm{c}}$ and $\mathcal{W}_{\mathrm{o}}$ are defined in
(\ref{W_c_def}) and (\ref{W_o_hat}), respectively. Then the closed-loop
system consisting of (\ref{sys_vary}) and (\ref{u_vary}) satisfies
$x(t)=0$ and $u(t)=0$ for all $t\geq3h$, meaning that the system is
fixed-time stable within a fixed time $3h\ge3(r-1)d+3\tau_{M}$.
\end{theorem}

\begin{proof}
For $t\in[3lh,2h+3lh]$, $l\in\mathbb{Z}_{0}^{+}$, in view of $g(t)=0$,
we have $\dot{x}(t)=Ax(t)$, whose solution is
\begin{equation}
x(t)=\mathrm{e}^{A(t-3lh)}x(3lh). \label{x_1}%
\end{equation}
For $t\in[3lh+2h,3(l+1)h]$, $l\in\mathbb{Z}_{0}^{+}$, we have $g(t)=1$.
Since $h\geq(r-1)d+\tau_{M}$, for $i=0,1,\ldots,r-1$,
$t-id-h-\tau(t-id-h)\in[3lh,3lh+2h]$.
It then follows from (\ref{sys_vary}) and (\ref{x_1}) that
\begin{align*}
y(t-id-h)
=C\mathrm{e}^{A(t-id-h-\tau(t-id-h)-3lh)}x(3lh).
\end{align*}
Therefore, by $K(t)=\bar K(t)M^{\mathrm{T}}$ and
(\ref{W_o_hat}), the closed-loop system can be represented as
\begin{align}
\dot{x}(t)
={}&Ax(t)+B\bar K(t)\mathcal{W}_{\mathrm{o}}(d,t)
\mathrm{e}^{A(t-h-3lh)}x(3lh). \label{x_dot12}%
\end{align}
By the variation-of-constants formula, the solution of (\ref{x_dot12}) is
\begin{equation}
x(t)=\mathit{\Delta}(t)x(3lh), \label{x_2}%
\end{equation}
where
\begin{align*}
\mathit{\Delta}(t)
={}&\mathrm{e}^{A(t-3lh)}
+\int_{3lh+2h}^{t}\mathrm{e}^{A(t-v)}
B\bar K(v)\mathcal{W}_{\mathrm{o}}(d,v)\\
&\qquad\times
\mathrm{e}^{A(v-h-3lh)}\mathrm{d}v.
\end{align*}
For $t=3(l+1)h$, substituting (\ref{k123}) into the above equation yields $\mathit{\Delta}(3(l+1)h)=0$
where the definition of $W_{\mathrm{c}}$ in (\ref{W_c_def}) has been used.
Thus, $x(3(l+1)h)=0$ for all $l\in\mathbb{Z}_{0}^{+}$. For
$t\in[3h,6h]$, it follows from (\ref{x_1}) and (\ref{x_2}) that $x(t)=0$.
Hence, from (\ref{sys_vary}) and (\ref{u_vary}), we have $u(t)=0$.
Repeating the above process gives $x(t)=0$ and $u(t)=0$ for all $t\geq3h$.
\end{proof}

For the chain of integrators, i.e., system (\ref{sys_vary}) with
$(A,B,C)$ defined in (\ref{A})--(\ref{C}) under $p=1$, we have $n=r$, and
$A_i=0$, $i=0,1,\ldots,n-1$. 
Then
\[
C\mathrm{e}^{-A\theta_j(d,t)}
=
\left[
1,-\theta_j(d,t),
\frac{\theta_j^2(d,t)}{2!},\ldots,
\frac{(-\theta_j(d,t))^{n-1}}{(n-1)!}
\right],
\]
where for $j=0,1,\ldots,n-1$, $\theta_j(d,t)=jd+\tau(t-jd-h)$. Hence,
$
\operatorname{col}_{j=0}^{n-1}
\left\{
C\mathrm{e}^{-A\theta_j(d,t)}
\right\}
$
is a scaled Vandermonde matrix and is nonsingular whenever
$\theta_0(d,t),\ldots,\theta_{n-1}(d,t)$ are pairwise distinct.

In particular, if (\ref{slow_delay}) holds, then, for $j=0,1,\ldots,n-2$,
$\theta_{j+1}(d,t)-\theta_j(d,t)
=d+\tau(t-(j+1)d-h)
-\tau(t-jd-h)
\geq(1-\mu)d>0$.
Moreover,
$(1-\mu)d
\leq
\theta_{j+1}(d,t)-\theta_j(d,t)$,
and
$
jd\leq\theta_j(d,t)\leq jd+\tau_M$.
Thus, for every fixed $d>0$, the nodes $\theta_j(d,t)$ belong to a bounded set and
are uniformly separated with respect to $t\geq2h$. Consequently, the
scaled Vandermonde matrix and its inverse are uniformly bounded with
respect to $t\geq2h$. Since $M^{\mathrm T}$ is nonsingular, it follows from (\ref{W_o_hat}) that
$
\inf_{t\geq2h}
\sigma_{\min}
\left(
\mathcal W_{\mathrm o}(d,t)
\right)>0
$
for every fixed $d>0$. Therefore, condition
(\ref{W_o_uniform}) holds for every fixed $d>0$, and we obtain the
following corollary.

\begin{corollary}
\label{coro1}
Consider the chain of integrators, i.e., system (\ref{sys_vary}) with
$(A,B,C)$ defined in (\ref{A})--(\ref{C}) under $p=1$, and hence $n=r$, and
$A_{i}=0$, $i=0,1,\ldots,n-1$. Assume that (\ref{slow_delay}) holds.
For any $d>0$, choose $h\geq(n-1)d+\tau_M$. Consider the time-varying
static output feedback controller (\ref{u_vary}). Then the corresponding
closed-loop system is fixed-time stable within a fixed time $3h$.
\end{corollary}

\subsection{Periodic Sampled-Data Measurements}

In this part, we assume that only sampled measurements $y(t_{k})$ are
available to the controller, where $t_{k}=kd$, $k\in\mathbb{Z}_{0}^{+}$. Equivalently, the held sampled output can be written as
$y_s(t)=Cx(t-\tau_s(t))$, where
$\tau_s(t)=t-t_k\in[0,d)$ for $t\in[t_k,t_{k+1})$.
Thus, $\tau_s(t)$ is a periodic sawtooth measurement delay. For
later use, define 
\begin{equation}
P(d,t)=[P_{0}^{\mathrm{T}}(d,t),P_{1}^{\mathrm{T}}(d,t),\ldots,
P_{r-1}^{\mathrm{T}}(d,t)]^{\mathrm{T}}\in\mathbb{R}^{n\times n},
\label{P_def}
\end{equation}
 with $P_{i}(d,t)\in\mathbb{R}^{p\times n}$ given by
\begin{align}
P_{i}(d,t)=  &  -\varphi_{i}(id) CA^{i+1}\int_{t_{k}}^{t}\mathrm{e}^{A(s-t-id)
}\mathrm{d}s\nonumber\\
&  +\int_{t_{k}-h}^{t-h}\int_{0}^{id}\varphi_{i}^{^{\prime}}(s) CA^{i+1}%
\mathrm{e}^{A(u-s-t+h) }\mathrm{d}s\mathrm{d}u\nonumber\\
&  -\int_{0}^{id}\varphi_{i}(s) CA^{i+1}\mathrm{e}^{-As}\mathrm{d}s.
\label{P_i}%
\end{align}
We first prove
that there exists a sufficiently small positive constant $d_{\ast}$, such that
for $d\in(0,d_{\ast})$, $W_{\mathrm{o}}+P(d,t)$ is invertible. For the last
term in $P_{i}(d,t)$, we have
$
\lim_{d\rightarrow0^{+}}\int_{0}^{id}\varphi_{i}(s) CA^{i+1}\mathrm{e}%
^{-As}\mathrm{d}s =0$.

For the remaining two terms, since $t\in[t_{k},t_{k+1})$, we have $0\leq
t-t_{k}<d$. Hence, the first term in $P_{i}(d,t)$ also converges uniformly to
zero as $d\to0^{+}$. Moreover, the second term in $P_{i}(d,t)$ converges
uniformly to zero as $d\to0^{+}$ because the outer integration interval has
length less than $d$, while the inner integral is uniformly bounded.
Therefore, $\lim_{d\to0^{+}}\sup_{t\geq0}|P(d,t)|=0$. Since $W_{\mathrm{o}}$
is invertible, there exists a sufficiently small constant $d_{\ast}>0$ such
that for all $d\in(0,d_{\ast})$ and all $t\geq0$, $W_{\mathrm{o}}+P(d,t)$ is invertible.

With the above definition, we can state the following result.
\begin{theorem}
\label{sample_data}Consider the system (\ref{sys_2}) under Assumptions \ref{ass1}-\ref{ass2} and
$w_{1}(t)=0,w_{2}(t)=0$. Given $h=Nd$, $N\in\mathbb{N}$, and $N\ge r-1$,
choose small enough $d>0$ such that $W_{\mathrm{o}}+P(d,t)$ is invertible for all $t\geq0$. Consider the
time-varying static output feedback controller
\begin{equation}
\begin{aligned}
u(t) &=\sum_{i=0}^{r-1} g(t) K_{i}(t)\, y(t_{k-i}-h) \\
     &=\sum_{i=0}^{r-1} g(t) K_{i}(t)\, y(t_{k-i-N}), \quad t\in[t_{k}%
,t_{k+1}),
\end{aligned}
\label{u_control_5}
\end{equation}
where $t_{k}=kd$, $K_{i}(t)\in\mathbb{R}^{m\times p}$ for $i=0,1,\ldots,r-1$ are the block
columns of $K(t)\in\mathbb{R}^{m\times n}$, as defined in (\ref{K_block}), satisfying
$K(t)=\bar K(t) M^{\mathrm{T}}$ with $M\in\mathbb{R}^{n\times n}$ defined in
(\ref{Mij}) and $\bar K(t)$ given by
\begin{equation}
\bar{K}(t)=-S(t)B^{\mathrm{T}}\mathrm{e}^{A^{\mathrm{T}}(3lh-t)}W_{\mathrm{c}%
}\mathrm{e}^{-A(t-h-3lh)}(W_{\mathrm{o}}+P(d,t))^{-1}, \label{K_hat1}%
\end{equation}
where $W_{\mathrm{c}}$ and $W_{\mathrm{o}}$ are defined in (\ref{W_c_def}) and
(\ref{Woo}), respectively. Then the closed-loop system (\ref{sys_2}) and
(\ref{u_control_5}) satisfies the condition that $x(t)  =0$ and
$u(t)=0$, for $t\geq3h$, meaning that the system is fixed-time stable.
\end{theorem}

\begin{proof}
In view of (\ref{y_i_hat}) and $K(t)=\bar{K}(t)M^{\mathrm{T}}$, the controller
(\ref{u_control_5}) can be represented as
\[
u(t)=%
%TCIMACRO{\dsum _{i=0}^{r-1}}%
%BeginExpansion
{\displaystyle\sum_{i=0}^{r-1}}
%EndExpansion
g(t)\bar{K}_{i}(t)\tilde{y}_{i}(t_{k}-h).
\]
Similar to \cite{anton18auto}, from (\ref{y_0_hat}) and (\ref{y_h}), it follows that for $i=0,1,\dots,r-1$,
\begin{align*}
\tilde{y}_{i}(t_{k})=  &  \tilde{y}_{i}(t)-\int_{t_{k}}^{t}\dot{\tilde{y}}%
_{i}(s)\mathrm{d}s=y_{0}^{(i)}(t)\\
&  -\int_{t-id}^{t}\varphi_{i}(t-s)y_{0}^{(i+1)}(s)\mathrm{d}s-\int_{t_{k}%
}^{t}\dot{\tilde{y}}_{i}(s)\mathrm{d}s.
\end{align*}
In view of (\ref{y_h}), the closed-loop system (\ref{sys_2}) and (\ref{u_control_5}) can be
represented as
\[
\dot{x}(t)=Ax(t)+B%
%TCIMACRO{\dsum _{i=0}^{r-1}}%
%BeginExpansion
{\displaystyle\sum_{i=0}^{r-1}}
%EndExpansion
g(t)\bar{K}_{i}(t)\left(  y_{0}^{(i)}(t-h)+\delta_{i}(t)+\kappa_{i}(t)\right)
,
\]
where $\delta_{i}(t) =-\int_{t_{k}-h}^{t-h}\dot{\tilde{y}}_{i}(s)\mathrm{d}s$ and
\begin{align*}
\kappa_{i}(t)  =-\int_{t-id-h}^{t-h}\varphi_{i}(t-h-s)y_{0}^{(i+1)}%
(s)\mathrm{d}s.
\end{align*}
For $t\in\lbrack3lh,2h+3lh]$, $l\in%
%TCIMACRO{\U{2124} }%
%BeginExpansion
\mathbb{Z}
%EndExpansion
_{0}^{+}$, in view of $g(t)=0$, we have $\dot{x}(t)=Ax(t)$, whose solution can
be calculated as
\begin{equation}
x(t)=\mathrm{e}^{A(t-3lh)}x(3lh). \label{x_solu_5}%
\end{equation}
For $t\in\lbrack3lh+2h,3(l+1)h]$, $l\in%
%TCIMACRO{\U{2124} }%
%BeginExpansion
\mathbb{Z}
%EndExpansion
_{0}^{+}$, by noting $g(t)=1$ and (\ref{y_norm_def}), it follows
that
\begin{equation}
\dot{x}(t)=Ax(t)+B%
%TCIMACRO{\dsum _{i=0}^{r-1}}%
%BeginExpansion
{\displaystyle\sum_{i=0}^{r-1}}
%EndExpansion
\bar{K}_{i}(t)( CA^{i}x(t-h)+\delta_{i}(t)+\kappa_{i}(t)) . \label{x_dot_3}%
\end{equation}
Since $t-id-h\in\lbrack3lh,2h+3lh]$ and $t-h\in\lbrack3lh,2h+3lh]$, from
(\ref{y_norm_def}) and (\ref{x_solu_5}), we can obtain
\begin{align*}
&  \int_{t-id-h}^{t-h}\varphi_{i}(t-h-s)y_{0}^{(i+1)}(s)\mathrm{d}s\\
=  &  \int_{t-id-h}^{t-h}\varphi_{i}(t-h-s)CA^{i+1}\mathrm{e}^{A(s-3lh)}%
\mathrm{d}sx(3lh).
\end{align*}
According to (\ref{y_h}), we obtain
\begin{align*}
& \tilde{y}_{i}(t)= y_{0}^{(i)}(t)-\int_{t-id}^{t}\varphi_{i}(t-s)y_{0}%
^{(i+1)}(s)\mathrm{d}s\\
=  &  y_{0}^{(i)}(t)+\varphi_{i}(s)y_{0}^{(i)}(t-s)|_{0}^{id}-\int_{0}%
^{id}y_{0}^{(i)}(t-s)\mathrm{d}\varphi_{i}(s)\\
=  &  \varphi_{i}(id)y_{0}^{(i)}(t-id)-\int_{0}^{id}\varphi_{i}^{\prime
}(s)y_{0}^{(i)}(t-s)\mathrm{d}s,
\end{align*}
which, together with (\ref{y_norm_def}) and (\ref{x_solu_5}) and $t-d-h\leq t_{k}-h\leq t-h$, yields
\begin{align*}
&  \int_{t_{k}-h}^{t-h}\dot{\tilde{y}}_{i}(s)\mathrm{d}s = \int_{t_{k}%
-h}^{t-h}\varphi_{i}(id)y_{0}^{(i+1)}(s-id)\mathrm{d}s\\
&  -\int_{t_{k}-h}^{t-h}\int_{0}^{id}\varphi_{i}^{\prime}(s)y_{0}%
^{(i+1)}(u-s)\mathrm{d}s\mathrm{d}u\\
=  &  \int_{t_{k}-h}^{t-h}\varphi_{i}(id)CA^{i+1}\mathrm{e}^{A(s-id-3lh)}%
\mathrm{d}sx(3lh)\\
&  -\int_{t_{k}-h}^{t-h}\int_{0}^{id}\varphi_{i}^{\prime}(s)CA^{i+1}%
\mathrm{e}^{A(u-s-3lh)}\mathrm{d}s\mathrm{d}ux(3lh).
\end{align*}
Then (\ref{x_dot_3}) can be continued as
\[
\dot{x}(t)=Ax(t)+B\bar{K}(t)(W_{\mathrm{o}%
}+P(d,t))\mathrm{e}^{A(t-h-3lh)}x(3lh),
\]
where $P(d,t)$ is defined in (\ref{P_def}), whose solution can be calculated
as $x(t)=\mathit{\Delta}(t)x(3lh)$,
where%
\begin{align*}
  &\mathit{\Delta}(t)  =\mathrm{e}^{A(t-3lh)}\\
&  +\int_{2h+3lh}^{t}\mathrm{e}^{A(t-s)}B\bar{K}(s)(W_{\mathrm{o}%
}+P(d,s))\mathrm{e}^{A(s-h-3lh)}\mathrm{d}s.
\end{align*}
For $t=3(l+1)h$, we have $x(3(l+1)h)=\mathit{\Delta}(3(l+1)h)x(3lh)$. In view
of (\ref{K_hat1}), we obtain $\mathit{\Delta}(3(l+1)h)=0_{n}$, which implies
that $x(3(l+1)h)=0$ for $l\in%
%TCIMACRO{\U{2124} }%
%BeginExpansion
\mathbb{Z}
%EndExpansion
_{0}^{+}$. Similar to the proof of Theorem \ref{the4}, we obtain that $x(t)=0$
and $u(t)=0$, for $t\geq3h$.
\end{proof}

For the chain of integrators, i.e., system (\ref{sys_2}) with
$(A,B,C)$ defined in (\ref{A})--(\ref{C}) under $p=1$ and
$A_i=0$, $i=0,1,\ldots,n-1$, and
$w_1(t)=0,w_2(t)=0$, we have $n=r$. Moreover,
$CA^i=e_{1,i}^{\mathrm T}$ for $i=0,1,\ldots,n-1$,
and $CA^n=0$.
Furthermore, for any scalar $s$, $CA^{i+1}\mathrm{e}^{-As}$ has its first
$i+1$ entries equal to zero. Hence every term in
$P_{i}(d,t)$ defined in (\ref{P_i}), which contains $CA^{i+1}\mathrm{e}%
^{A(\cdot)}$, has zero entries in the first $i+1$ entries. Therefore, for
$i=0,1,\ldots,n-1$, $P_{i}(d,t)=[0_{1\times(i+1)},\ast,\ldots,\ast]
\in\mathbb{R}^{1\times n}$, and, since $CA^{n}=0$, $P_{n-1}(d,t)=0_{1\times
n}.$ It follows that $P(d,t)$ is strictly upper triangular. Clearly,
$W_{\mathrm{o}}+P(d,t)=I_{n}+P(d,t)$. Therefore, for any $d>0$, $W_{\mathrm{o}%
}+P(d,t)$ is invertible, which leads to the following corollary.

\begin{corollary}
Consider the chain of integrators, i.e., system (\ref{sys_2}) with $(A,B,C)$
defined in (\ref{A})--(\ref{C}) under $p=1$, and hence $n=r$, and
$A_{i}=0$, $i=0,1,\ldots,n-1$,
and $w_{1}(t)=0,w_{2}(t)=0$. Given any $d>0$ and $h=Nd,N\in\mathbb{N}$, $N\geq n-1$,
consider the time-varying static output feedback controller
(\ref{u_control_5}) under discrete-time measurements with $t_k=kd$. Then the corresponding closed-loop system is fixed-time stable within a fixed time $3h$.
\end{corollary}

\subsection{Systems with Nilpotent Matrices in the Presence of Unknown Constant Delays}

Consider the linear system with an unknown constant output delay
\begin{equation}
\left\{
\begin{array}
[c]{rl}%
\dot{x}\left(  t\right)   & =Ax\left(  t\right)  +Bu\left(  t\right)  ,\\
y\left(  t\right)   & =\left\{
\begin{array}
[c]{ll}%
Cx\left(  t-\tau_{0}\right)  , & t-\tau_{0}\geq0,\\
0, & t-\tau_{0}<0,
\end{array}
\right.
\end{array}
\right.  \label{sys_inte}%
\end{equation}
where $\tau_{0}$ is an unknown constant delay satisfying $\tau_{0}\in
\lbrack0,\tau_{M}]$ with $\tau_{M}$ being known.

For such a system, the delayed feedback method in \cite{zhou22tac} cannot be
used to construct a fixed-time observer in the presence of an unknown output
delay. Therefore, observer-based output feedback cannot solve the fixed-time
control problem for system (\ref{sys_inte}). In this subsection, we address
this problem using static output feedback. Before doing so, we make the
following assumption.

\begin{assumption}
\label{ass3}$A$ is nilpotent satisfying $A^{n}=0$.
\end{assumption}

\begin{theorem}
\label{the5}Consider the system (\ref{sys_inte}) under Assumptions
\ref{ass1}-\ref{ass3}. Given $h\geq(r-1)d+\tau_{M}$, choose $d>0$ such that
$W_{\mathrm{o}}+L(d)$ is invertible. Consider the smooth time-varying static
output feedback controller
\begin{equation}
u(t)=\sum_{i=0}^{r-1} g(t)K_{i}(t)\, y(t-id-h), \label{u_2}%
\end{equation}
where $K_{i}(t)\in\mathbb{R}^{m\times p}$ for $i=0,1,\ldots,r-1$ are the block
columns of $K(t)\in\mathbb{R}^{m\times n}$, as defined in (\ref{K_block}), satisfying
$K(t)=\bar{K}(t)M^{\mathrm{T}}$ with $M\in\mathbb{R}^{n\times n}$ and $\bar
K(t)$ being defined in (\ref{Mij}) and (\ref{K_hat}), respectively. Then the
closed-loop system (\ref{sys_inte}) and (\ref{u_2}) satisfies the condition
that $x(t)=0$ and $u(t)=0$, for $t\geq3nh$, meaning that the system is
fixed-time stable.
\end{theorem}

\begin{proof}
For $t\in\lbrack3lh,2h+3lh]$, $l\in%
%TCIMACRO{\U{2124} }%
%BeginExpansion
\mathbb{Z}
%EndExpansion
_{0}^{+}$, we have $g(t)=0$. Then the closed-loop system (\ref{sys_inte}) and
(\ref{u_2}) can be represented as $\dot{x}(t)=Ax(t)$, whose solution can be
given as
\begin{equation}
x(t)=\mathrm{e}^{A(t-3lh)}x(3lh). \label{x_dout1}%
\end{equation}
From (\ref{y_norm_def}), we obtain $y(t)=y_{0}(t-\tau_{0})$. Then for
$t\in\lbrack3lh+2h,3(l+1)h]$, $l\in%
%TCIMACRO{\U{2124} }%
%BeginExpansion
\mathbb{Z}
%EndExpansion
_{0}^{+}$, using $g(t)=1$, (\ref{y_norm_def}), (\ref{y_h}), and (\ref{K_Kbar_M_relation}), we have
\begin{align}
& \dot{x}(t)= Ax(t)+B%
%TCIMACRO{\dsum _{i=0}^{r-1}}%
%BeginExpansion
{\displaystyle\sum_{i=0}^{r-1}}
%EndExpansion
K_{i}(t)y_{0}(t-\tau_{0}-id-h)\nonumber\\
=  &  Ax(t)+B%
%TCIMACRO{\dsum _{i=0}^{r-1}}%
%BeginExpansion
{\displaystyle\sum_{i=0}^{r-1}}
%EndExpansion
\bar{K}_{i}(t)CA^{i}x(t-h-\tau_{0})+B%
%TCIMACRO{\dsum _{i=0}^{r-1}}%
%BeginExpansion
{\displaystyle\sum_{i=0}^{r-1}}
%EndExpansion
R_{i}(t), \label{x_dot1}%
\end{align}
where for $i=0,1,\ldots,r-1$,
\[
R_{i}(t)=-\bar{K}_{i}(t)\int_{t-id-h-\tau_{0}}^{t-h-\tau_{0}}\varphi
_{i}(t-h-\tau_{0}-s)y_{0}^{(i+1)}(s)\mathrm{d}s.
\]
Since $t-h-\tau_{0}\in\lbrack3lh,3lh+2h]$ and $t-id-h-\tau_{0}\in
\lbrack3lh,3lh+2h]$, then from (\ref{y_norm_def}) and (\ref{x_dout1}), we have
\begin{align*}
&  \int_{t-id-h-\tau_{0}}^{t-h-\tau_{0}}\varphi_{i}(t-h-\tau_{0}%
-s)y_{0}^{(i+1)}(s)\mathrm{d}s\\
=  &  \int_{t-id-h}^{t-h}\varphi_{i}(t-h-s)y_{0}^{(i+1)}(s-\tau_{0}%
)\mathrm{d}s\\
=  &  \int_{t-id-h}^{t-h}\varphi_{i}(t-h-s)CA^{i+1}x(s-\tau_{0})\mathrm{d}s\\
=  &  \int_{t-id-h}^{t-h}\varphi_{i}(t-h-s)CA^{i+1}\mathrm{e}^{A(s-3lh)}%
\mathrm{d}s\mathrm{e}^{-A\tau_{0}}x(3lh).
\end{align*}
Therefore, (\ref{x_dot1}) can be continued as
\[
\dot{x}(t)=Ax(t)+B\bar{K}(t)E(d)\mathrm{e}^{A(t-h-3lh)}\mathrm{e}^{-A\tau_{0}%
}x(3lh),
\]
where $E(d)$ is defined in (\ref{E_def}), whose solution $x(t)$ can be calculated as
\begin{equation}
x(t)=\mathit{\Delta}(t)x(3lh), \label{x_222}%
\end{equation}
where
\begin{align*}
\mathit{\Delta}(t)=  &  \mathrm{e}^{A(t-3lh)}+\int_{2h}^{t-3lh}\mathrm{e}%
^{A(t-3lh-v)}B\\
&  \times\bar{K}(v+3lh)E(d)\mathrm{e}^{A(v-h)}\mathrm{d}v\mathrm{e}%
^{-A\tau_{0}}.
\end{align*}
Therefore, for $t=3(l+1)h$, from (\ref{K_hat}), it follows that for $l\in%
%TCIMACRO{\U{2124} }%
%BeginExpansion
\mathbb{Z}
%EndExpansion
_{0}^{+}$,
$
x(3(l+1)h)=\mathit{\Delta}(3(l+1)h)x(3lh)=\mathrm{e}^{3Ah}(I_{n}%
-\mathrm{e}^{-A\tau_{0}})x(3lh)$,
iterating which yields for $l\in\mathbb{N}$,
$x(3lh)=(\mathrm{e}^{3Ah}(I_{n}-\mathrm{e}^{-A\tau_{0}}))^{l}x(0)$.
Since $A$ is nilpotent, we have
$I_{n}-\mathrm{e}^{-A\tau_{0}}=A\tau_{0}-(A\tau_{0})^{2}/2!%
+\cdots-((-A\tau_{0})^{n-1})/(n-1)!$,
which implies that $(I_{n}-\mathrm{e}^{-A\tau_{0}})^{n}=0$. Since both $\mathrm e^{3Ah}$ and $I_n-\mathrm e^{-A\tau_0}$ are polynomials in $A$, they commute. Thus
$x(3nh) =\mathrm{e}^{3Anh}(I_{n}-\mathrm{e}^{-A\tau_{0}})^{n}x(0)=0$.

For $t\in\lbrack3nh,(3n+2)h]$, from (\ref{x_dout1}) it follows that $x(t)=0$.
For $t\in\lbrack(3n+2)h,(3n+3)h]$, from (\ref{x_222}) and (\ref{u_2}) we have
$x(t)=0$ and $u(t)=0$. Repeating the above process, for all $t\geq3nh$, we can
also get that $x(t)=0$ and $u(t)=0$.
\end{proof}

When $p=1$, we have $n=r$. Since $A$ is nilpotent, we have
$\mathrm{e}^{Au}=\sum_{i=0}^{n-1}A^{i}u^{i}/i!$. With this,
for $q=0,1,\ldots,n-2$, $L_q(d)$ in (\ref{L_h_def}) can be
simplified to
$
L_{q}(d)=-C
\sum_{i=0}^{n-2-q}
\frac{A^{i+q+1}}{i!}
\int_{-q d}^{0}\varphi_{q}\left(-u\right)u^{i}\mathrm{d}u.
$
Moreover, $L_{n-1}(d)=0$ because $A^n=0$.
Assumption \ref{ass3} together with the
canonical form (\ref{A}) implies that $A_i=0$, $i=0,1,\ldots,n-1$. In this
case, $\det(W_{\mathrm{o}}+L(d))=1$ for any $d>0$. Therefore, we can obtain
the following corollary of Theorem \ref{the5}.

\begin{corollary}
\label{cor1}Consider the system (\ref{sys_inte}) with $(A,B,C)$ defined in
(\ref{A})--(\ref{C}) under Assumptions \ref{ass1}--\ref{ass3} and $p=1$.
For any $d>0$, choose $h\geq(n-1)d+\tau_M$. Consider the smooth time-varying static
output feedback controller (\ref{u_2}). Then the corresponding closed-loop
system is fixed-time stable within a fixed time $3nh$.
\end{corollary}
Furthermore, based on Theorem \ref{the3}, we can obtain the following
simple result for chain of integrators with the single output $p=1$. Under Assumption
\ref{ass1}, $p=1$ implies that $n=r$. Moreover, under Assumption
\ref{ass3} and the canonical form (\ref{A}), we have
$A_i=0$, $i=0,1,\ldots,n-1$.
\begin{theorem}
\label{the6}Consider the chain of integrators with the single output ($p=1$) governed by system (\ref{sys_inte}) with $(A,B,C)$ in the canonical form
(\ref{A})--(\ref{C}) under Assumptions \ref{ass1}--\ref{ass3}. For any $d>0$, choose $h\geq(n-1)d+\tau_M$. Consider the smooth time-varying static output
feedback controller (\ref{output}), where $\mathcal{K}(t)=\mathcal{\bar
K}(t)\mathcal{M}^{-1}$ and
\begin{equation}
\mathcal{\bar K}(t)=-S(t)B^{\mathrm{T}}\mathrm{e}^{A^{\mathrm{T}}(3lh-t)}
W_{\mathrm{c}}\,\mathrm{e}^{-A(t-3lh-h)}, \label{K_hat11}%
\end{equation}
with $\mathcal{M}$ and $W_{\mathrm{c}}$ defined in (\ref{Mm}) and
(\ref{W_c_def}), respectively. Then the corresponding closed-loop system is fixed-time stable within a fixed time $3nh$.
\end{theorem}

\begin{proof}
The closed-loop system consisting of (\ref{sys_inte}) and (\ref{output}) can
be written as
\begin{equation}
\dot{x}(t)=Ax(t)+%
%TCIMACRO{\dsum _{i=0}^{n-1}}%
%BeginExpansion
{\displaystyle\sum_{i=0}^{n-1}}
%EndExpansion
Bg(t)\mathcal{K}_{i}(t)x_{0}(t-d_{i}-h-\tau_{0}) . \label{x_dot_clo_1}%
\end{equation}
According to (\ref{x_0}), we obtain
\begin{align*}
&
%TCIMACRO{\dsum _{i=0}^{n-1}}%
%BeginExpansion
{\displaystyle\sum_{i=0}^{n-1}}
%EndExpansion
\mathcal{K}_{i}(t)x_{0}(t-d_{i}-h-\tau_{0})\\
&  =%
%TCIMACRO{\dsum _{j=0}^{n-1}}%
%BeginExpansion
{\displaystyle\sum_{j=0}^{n-1}}
%EndExpansion
\mathcal{\bar{K}}_{j}(t)x_{j}(t-h-\tau_{0})+%
%TCIMACRO{\dsum _{i=0}^{n-1}}%
%BeginExpansion
{\displaystyle\sum_{i=0}^{n-1}}
%EndExpansion
\mathcal{K}_{i}(t)\mathcal{R}_{i}(t-h-\tau_{0}),
\end{align*}
where $\mathcal{R}_{i}(t)$ is defined in (\ref{R_i01}), substituting which into (\ref{x_dot_clo_1}) yields
\begin{align}
\dot{x}(t)  =&Ax(t)+Bg(t)\mathcal{\bar{K}}(t)x(t-h-\tau_{0})\nonumber\\
& +\sum_{i=0}^{n-1}Bg(t)\mathcal{K}_{i}(t)\mathcal{R}_{i}(t-h-\tau
_{0}).\label{x_dot2_1}%
\end{align}

For $t\in\lbrack3lh,2h+3lh]$, $l\in%
%TCIMACRO{\U{2124} }%
%BeginExpansion
\mathbb{Z}
%EndExpansion
_{0}^{+}$, in view of $g(t)=0$, (\ref{x_dot2_1}) reduces to $\dot
{x}(t)=Ax(t)$, whose solution can be expressed as
\begin{equation}
x(t)=\mathrm{e}^{A(t-3lh)}x(3lh). \label{x_123}%
\end{equation}
For $t\in\lbrack2h+3lh,3(l+1)h]$, $l\in%
%TCIMACRO{\U{2124} }%
%BeginExpansion
\mathbb{Z}
%EndExpansion
_{0}^{+}$, by noting $g(t)=1$ and (\ref{x_123}), (\ref{x_dot2_1}) can be
rewritten as
\begin{align}
\dot{x}(t)=  &  Ax(t)+B\mathcal{\bar{K}}(t)\mathrm{e}^{A(t-3lh-h-\tau_{0}%
)}x(3lh)\nonumber\\
&  +%
%TCIMACRO{\dsum _{i=0}^{n-1}}%
%BeginExpansion
{\displaystyle\sum_{i=0}^{n-1}}
%EndExpansion
B\mathcal{K}_{i}(t)\mathcal{R}_{i}(t-h-\tau_{0}). \label{x_t_cl2_1}%
\end{align}
Since $t-h-\tau_{0}\in\lbrack3lh,2h+3lh]$ and $t-d_{i}-h-\tau_{0}\in
\lbrack3lh,2h+3lh]$ for $d_{i}+\tau_{0}\leq h$, $x_{n}(s)$ in $\mathcal{R}%
_{i}(t-h-\tau_{0})$ can be calculated as
$x_{n}(t)  =\dot{x}_{n-1}(t)=e_{1,n-1}^{\mathrm{T}}\dot{x}(t)
 =e_{1,n-1}^{\mathrm{T}}Ax(t)=e_{1,n-1}^{\mathrm{T}}A\mathrm{e}%
^{A(t-3lh)}x(3lh)$.
For $p=1$, according to Assumption \ref{ass3} with the matrix $A$ in the canonical form (\ref{A}), we obtain $A_{i}=0$,
$i=0,1,\ldots,n-1$ and thus $e_{1,n-1}^{\mathrm{T}}A\mathrm{e}^{A(t-3lh)}%
=0_{1\times n}$. With this, (\ref{x_t_cl2_1}) can be continued as
$\dot{x}(t)=Ax(t)+B\mathcal{\bar{K}}(t)\mathrm{e}^{A(t-h-\tau_{0})}%
\mathrm{e}^{-3lhA}x(3lh)$,
whose solution can be expressed as
\begin{align*}
x(t)=  &  \mathrm{e}^{A(t-3lh)}x(3lh)\\
&  +\int_{2h+3lh}^{t}\mathrm{e}^{A(t-s)}B\mathcal{\bar{K}}(s)\mathrm{e}%
^{A(s-h-\tau_{0})}\mathrm{d}s\mathrm{e}^{-3lhA}x(3lh).
\end{align*}
Then for $t=3(l+1)h$, in view of (\ref{K_hat11}), we have
$x(3(l+1)h)=\mathrm{e}^{3Ah}(I_{n}-\mathrm{e}^{-A\tau_{0}})x(3lh),l\in%
%TCIMACRO{\U{2124} }%
%BeginExpansion
\mathbb{Z}
%EndExpansion
_{0}^{+}$. Similar to the proof of Theorem \ref{the5}, we obtain that $x(t)=0$
and $u(t)=0$, for $t\geq3nh$.
\end{proof}

\begin{remark}
Consider the chain of integrators with the single output ($p=1$) governed by system (\ref{sys_inte}) with $(A,B,C)$ in the canonical form
(\ref{A})--(\ref{C}) under Assumptions \ref{ass1}--\ref{ass3}, the controller (\ref{K_hat11}) in Theorem
\ref{the6} is simpler than the controller (\ref{u_2}) in Corollary
\ref{cor1}, since there is no need to calculate $\mathcal{L}(d)$
defined in (\ref{calL_def}). Moreover, for any $d>0$, one can choose
$h\geq(n-1)d+\tau_M$, whereas
\cite{emilia17auto,anton18auto} consider systems without measurement
delays and require the artificial delay $d$ to be sufficiently small.

We show next that the controller gain can be made arbitrarily small by increasing $h$, thereby potentially improving the attenuation of measurement noise, and examine the dependence of the gain on $h$. Since
$d_i=id$, $i=0,1,\ldots,n-1$, the matrix $\mathcal{M}$ defined in
(\ref{Mm}) can be factorized as
$\mathcal{M}
=
\mathcal{M}_0
\operatorname{diag}
(1,d,\ldots,d^{n-1})$,
where $\mathcal{M}_0$ is independent of $d$. Hence,
$\mathcal{M}^{-1}
=
\operatorname{diag}
(1,d^{-1},\ldots,d^{-(n-1)})
\mathcal{M}_0^{-1}$,
which implies that $\vert\mathcal{M}^{-1}\vert$ remains bounded as
$d\rightarrow\infty$.
In addition to the conditions in Definition~\ref{def1}, for each \(h>0\), we choose \(S(s)=\bar{S}(s/h)\), \(s\in[2h,3h]\), where \(\bar{S}\) is independent of \(h\), nonnegative on \([2,3]\), and strictly positive at some point in \((2,3)\). Then, \(\sup_{h>0}|S|<\infty\).

For the chain of $n$ integrators, by using the change of
variables $s=h\theta$, $W_{\mathrm{c}}^{-1}$ defined in (\ref{W_c_def}) can be rewritten as $W_{\mathrm{c}}^{-1}=
hD_hG_{\mathrm{c}}D_h$,
where $D_h=\operatorname{diag}
(h^{n-1},h^{n-2},\ldots,h,1)$,
$
\phi(\theta)
=
\left[
(-\theta)^{n-1}/(n-1)!,
(-\theta)^{n-2}/(n-2)!,
\ldots,-\theta,1
\right]^{\mathrm T},
$
and
$G_{\mathrm{c}}=
\int_2^3
\phi(\theta)\bar{S}(\theta)
\phi^{\mathrm T}(\theta)\,\mathrm{d}\theta$.
Since $\bar S$ is continuous and strictly positive on a nonempty
subinterval of $(2,3)$, we have $G_{\mathrm{c}}>0$. Consequently,
$
W_{\mathrm{c}}
=
\frac{1}{h}
D_h^{-1}G_{\mathrm{c}}^{-1}D_h^{-1}$.
It follows that $\vert W_{\mathrm{c}}\vert
= O(h^{-1})$ as $h\rightarrow\infty$,
whereas $\vert W_{\mathrm{c}}\vert
=
O\bigl(h^{-(2n-1)}\bigr)$ as $h\rightarrow0^{+}$.
Let $\theta=(t-3lh)/h\in[2,3]$. Then $
B^{\mathrm T}\mathrm{e}^{A^{\mathrm T}(3lh-t)}
=
\phi^{\mathrm T}(\theta)D_h
$
and
$
\mathrm{e}^{-A(t-3lh-h)}
=
D_h\mathrm{e}^{-A(\theta-1)}D_h^{-1}$.
Substituting the above expressions into (\ref{K_hat11}) yields
$
\bar{\mathcal K}(t)
=
-\frac{S(t)}{h}
\phi^{\mathrm T}(\theta)
G_{\mathrm{c}}^{-1}
\mathrm{e}^{-A(\theta-1)}
D_h^{-1}.
$
Therefore, $\vert\bar{\mathcal K}(t)\vert
=
O(h^{-1})$ as $h\rightarrow\infty$,
whereas $\vert\bar{\mathcal K}(t)\vert
=
O(h^{-n})$ as $h\rightarrow0^{+}$.
Since $\mathcal K(t)
=
\bar{\mathcal K}(t)\mathcal{M}^{-1}$
and $\vert\mathcal{M}^{-1}\vert$ remains bounded as
$d\rightarrow\infty$, it follows that
$
\vert\mathcal K(t)\vert
\leq
\vert\bar{\mathcal K}(t)\vert
\vert\mathcal{M}^{-1}\vert
=O(h^{-1})$,
as $h\rightarrow\infty$.
Thus, by choosing $d$ arbitrarily large and increasing $h$
accordingly so that $h\geq(n-1)d+\tau_M$, the actual controller
gains can be made arbitrarily small. This endows the controller with
a low-gain property that may efficiently attenuate measurement noise.
\end{remark}
\section{\label{sec4}Numerical Examples}

\subsection{Chain of Two Integrators under Matched Disturbances}
\begin{figure}[t]
\centering \includegraphics[scale=0.5]{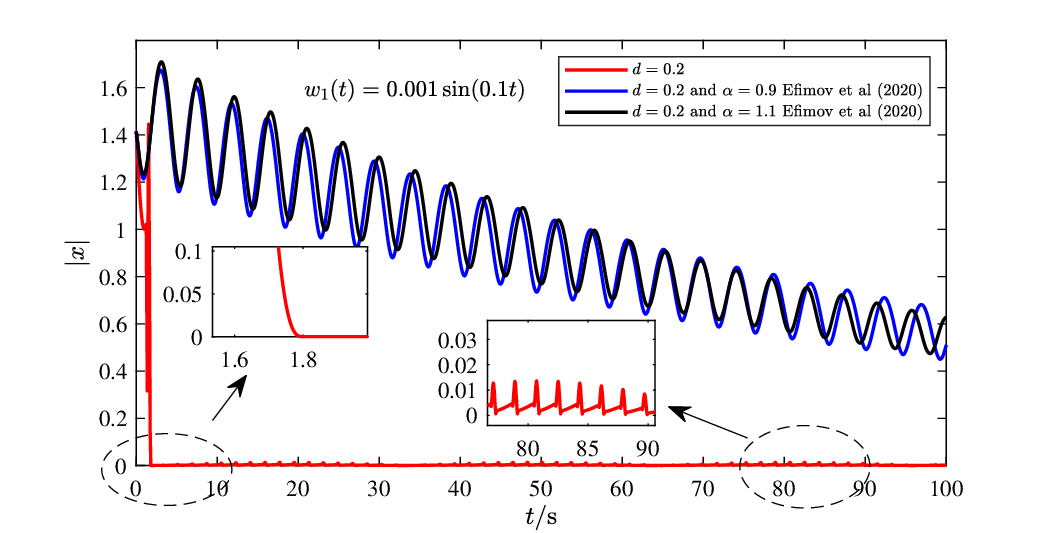}\caption{State trajectories of
the chain of two integrators}%
\label{figure0}%
\end{figure}
Here we compare our results with those in \cite{Efimov20auto}
(accelerated stabilization) for the chain of two integrators. From
(\ref{L_h_def}), we can calculate $\det(W_{\mathrm{o}}+L(d))=1$, which implies
that for any $d>0$, $W_{\mathrm{o}}+L(d)$ is invertible. In contrast, in
\cite{Efimov20auto}, the difference step size $d$ must satisfy $d\in
(5\times10^{-4},0.2]$. According to \cite{Efimov20auto}, the
matched disturbances are chosen as $w_{1}(t)=0.001\sin(0.1t)$, $w_{2}(t)=0$, and
$d=0.2$. For simulation, take $x(0)=[1,-1]^{\mathrm{T}}$. Take $S(t)=\sin
^{2}(\pi t/h)$ and $h=3d\geq 2d$. We conducted simulations using Theorem \ref{the_2}
and Theorem \ref{the3}, respectively. It was found that the two methods yield
exactly the same performance under the same $d$. The state trajectories are
plotted in Figure \ref{figure0}, where two different values of the controller
parameter $\alpha$ from \cite{Efimov20auto} are selected.
Simulation results show that our method has a faster convergence rate and our
controller leads to an ultimate bound of $0.02$ over the final time interval
$[1.8,100]$, meaning prescribed-time ISS is achieved with $T=3h=1.8$.

\subsection{Chain of Three Integrators}

\subsubsection{Unknown Constant Output Delays}

Consider the chain of three integrators under unknown constant output delays.
From (\ref{L_h_def}), we can calculate $\det(W_{\mathrm{o}}+L(d))=1$, which
implies that for any $d>0$, $W_{\mathrm{o}}+L(d)$ is invertible. For
simulation, similar to \cite{anton18auto}, take $x(0)=[1,-1,1]^{\mathrm{T}}$.
Take $S(t)=\sin^{2}(\pi t/h)$, $d=0.2$, and $h=3d$. We conducted simulations using
Theorem \ref{the5} and Theorem \ref{the6}, respectively. It was found that the
two methods yield exactly the same performance under the same $d$ and
$\tau_{0}$. The state trajectories are plotted in Figure \ref{figure1}.
In the simulations, we take $\tau_M=d$ and $\tau_0=d$.
Simulation results show that, for the all considered unknown output delays
$\tau_0\in[0,\tau_M]$ with $\tau_M\leq d$, the state converges
to zero no later than the fixed time $T=3nh=27d$.

\begin{figure}[t]
\centering \includegraphics[scale=0.5]{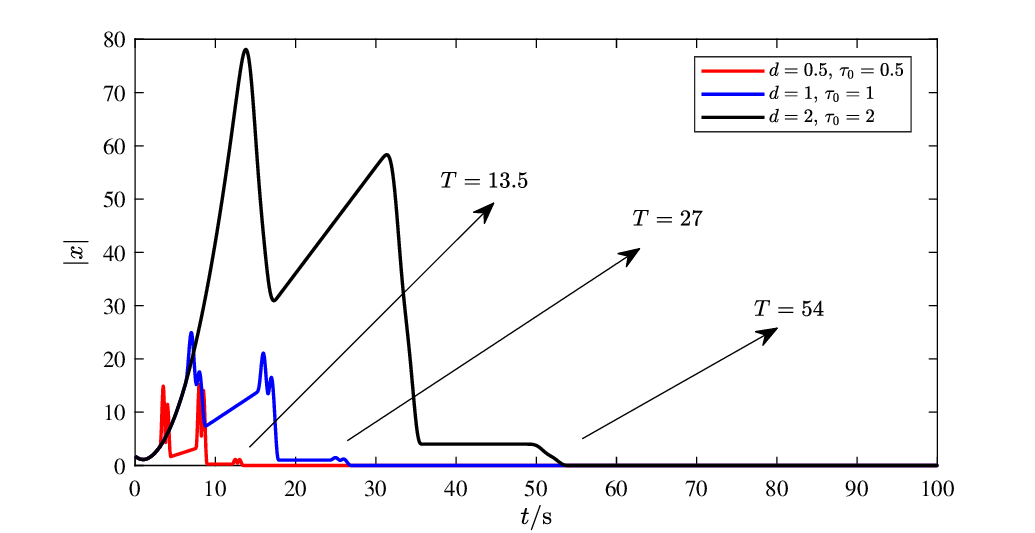}\caption{State trajectories of
the chain of three integrators with unknown constant output delays}%
\label{figure1}%
\end{figure}

\subsubsection{Comparison with Existing Results}

For systems of integrators without measurement delay, the difference step size
$d$ chosen to stabilize the system in the traditional static output feedback
methods \cite{emilia17auto} and \cite{anton18auto} is not arbitrary; it is
required to lie within a relatively small range, for example, $d\in(0,0.85]$
in \cite{emilia17auto} and $d\in(0,2.529]$ in \cite{anton18auto}. In contrast,
the method proposed in this paper does not have this limitation; any $d>0$ is
admissible, making our method more widely applicable.

\noindent\textbf{Nominal-system comparison.} In the following, we perform two comparisons in simulation: the first is with
asymptotic or exponential convergence, i.e., \cite{emilia17auto,anton18auto}
and \cite{nicule04tac}; the second is with practical fixed-time convergence,
i.e., \cite{Nekhoroshikh22auto}. Take $S(t)=\sin^{2}(\pi t/h)$ and $h=3d$. For
comparison with \cite{emilia17auto,nicule04tac,anton18auto}, similar to
\cite{anton18auto}, take $x(0)=[1,-1,1]^{\mathrm{T}}$. The state trajectories
are plotted in Figure \ref{figure4}. For comparison with
\cite{Nekhoroshikh22auto}, similar to \cite{Nekhoroshikh22auto}, take
$x(0)=[0,1,-0.5]^{\mathrm{T}}\times10^{8-2i}, \quad i=0,1,2,3$. The state
trajectories are plotted in Figures \ref{figure6} and \ref{figure5}. The above
simulation results show that for any $d>0$, the state converges to zero at the
prescribed-time $T=3h=9d$. Furthermore, our method achieves a faster convergence
rate compared to those in \cite{emilia17auto,nicule04tac,anton18auto} and
\cite{Nekhoroshikh22auto}. \begin{figure}[t]
\centering \includegraphics[scale=0.5]{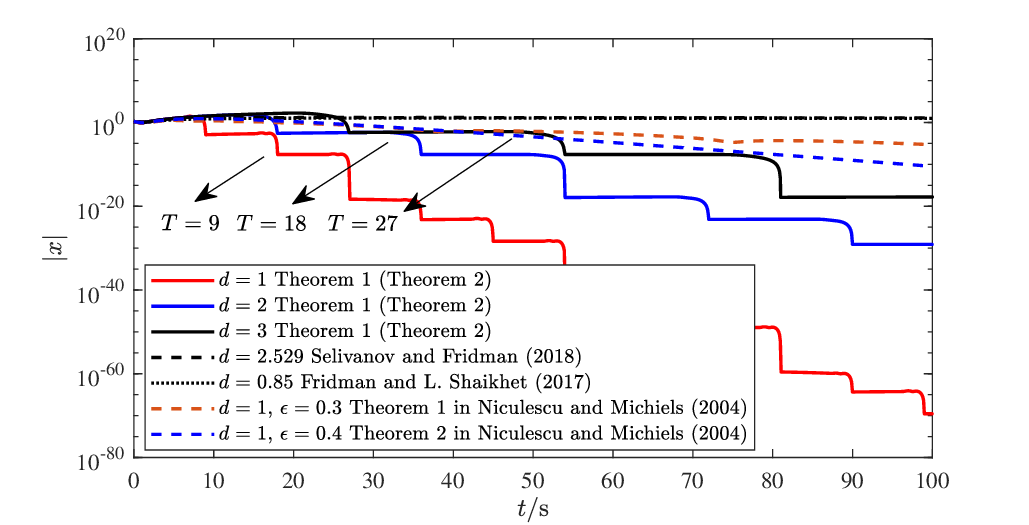}\caption{State trajectories of
the chain of three integrators without measurement delay}%
\label{figure4}%
\end{figure}\begin{figure}[t]
\centering \includegraphics[scale=0.5]{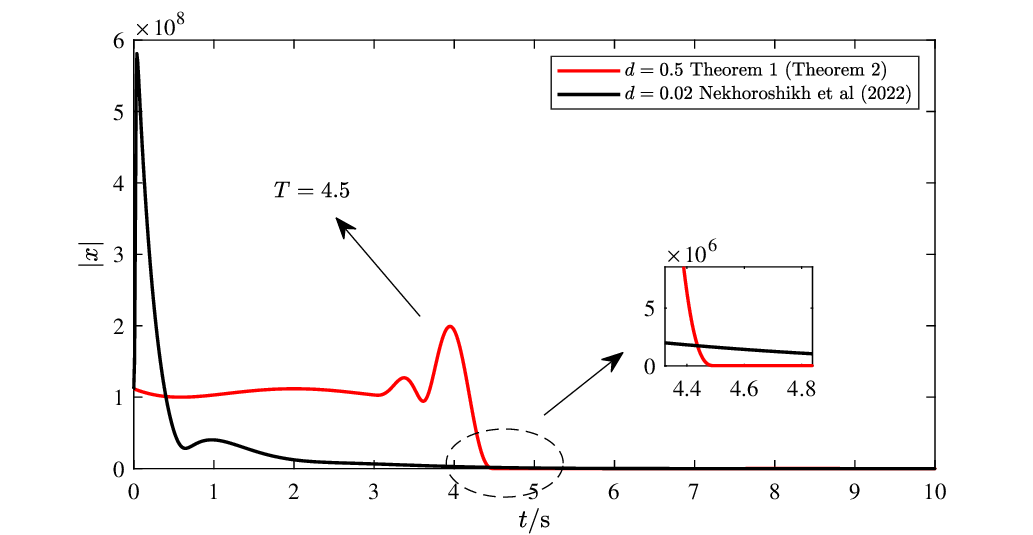}\caption{State trajectories of
the chain of three integrators under $x(0)=[0,1,-0.5]^{\mathrm{T}}\times
10^{8}$}%
\label{figure6}%
\end{figure}
\begin{figure}[t]
\centering \includegraphics[scale=0.5]{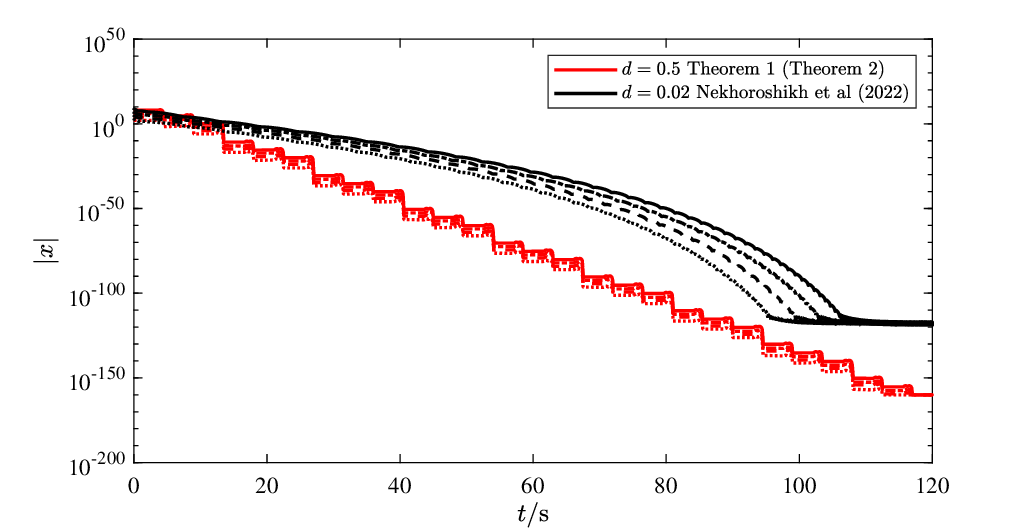}\caption{State trajectories of
the chain of three integrators under $x(0)=[0,1,-0.5]^{\mathrm{T}}%
\times10^{8-2i}, i=0,1,2,3$}%
\label{figure5}%
\end{figure}

\noindent\textbf{Comparison under measurement noise and unknown constant measurement delay.}
As pointed out in Remark~\ref{rem1}, a smaller $d$ leads to high-gain
characteristics of the feedback gain, thereby amplifying the measurement noise
of the system. Here, we compare our result with that in \cite{anton18auto} and
choose the largest allowable value of $d$ for the method in \cite{anton18auto},
i.e., $d=2.529$. For simulation, we take the initial condition
$x(0)=[1,-1,1]^{\mathrm{T}}$ and choose the measurement noise as
$w_2(t)=0.1\sin(10t)$. Moreover, according to our simulation tests, the method
in \cite{anton18auto} can tolerate only an unknown constant measurement delay
up to $\tau_0=0.5$, whereas our method allows an arbitrarily large $d>0$, provided that $h\geq(n-1)d+\tau_M$. In this simulation, we take $d=5$, $\tau_M=5$, and $h=3d=15$,
under which our method can tolerate the unknown constant measurement delay
$\tau_0=5\in[0,\tau_M]$. The simulation results are shown in Figure~\ref{figure9}, which
indicate that the proposed controller can tolerate a larger unknown constant
measurement delay and achieve a smaller steady-state error in the presence of
measurement noise.
\begin{figure}[t]
\centering \includegraphics[scale=0.5]{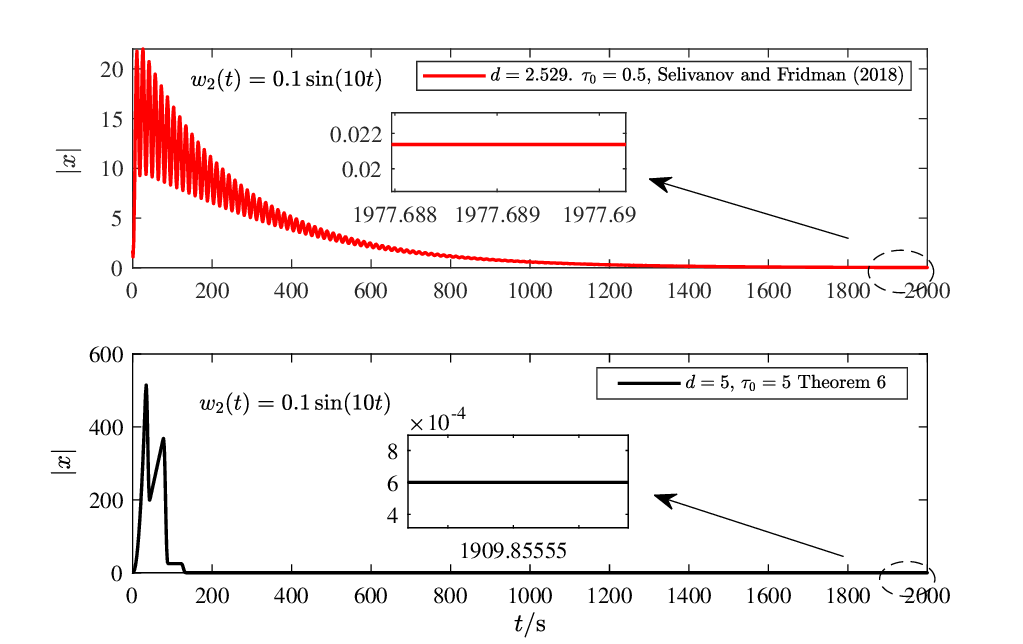}\caption{State trajectories of
the chain of three integrators with measurement noise and unknown constant time delays}%
\label{figure9}%
\end{figure}

\subsection{Spacecraft Rendezvous Control System}

We consider the linearized spacecraft rendezvous control system, namely, C-W
equation \cite{zhou17tcst,zhou22tac}
\begin{equation}
\left[
\begin{array}
[c]{c}%
\ddot{X}\\
\ddot{Z}\\
\ddot{Y}%
\end{array}
\right]  =\left[
\begin{array}
[c]{c}%
2\omega_{\mathrm{o}}\dot{Z}\\
3\omega_{\mathrm{o}}^{2}Z-2\omega_{\mathrm{o}}\dot{X}\\
-\omega_{\mathrm{o}}^{2}Y
\end{array}
\right]  +\left[
\begin{array}
[c]{c}%
a_{X}\\
a_{Z}\\
a_{Y}%
\end{array}
\right]  , \label{sys11}%
\end{equation}
where $\omega_{\mathrm{o}}=7.2722\times10^{-5}\mathrm{rad}/\mathrm{s}$ and
$a=[a_{X},a_{Z},a_{Y}]^{\mathrm{T}}$ is the acceleration vector due to thrust
forces on the chaser. From (\ref{sys11}), it can be observed that the $X$-$Z$
dynamics and the $Y$ dynamics are decoupled. Thus, they can be designed separately.
Here we consider three cases: no measurement delay (Theorem \ref{the_2}), slowly time-varying delay (Theorem \ref{the4}), and periodic sampled-data measurements (Theorem \ref{sample_data}). Here we take slowly time-varying delay $\tau(t)=\sin^2(0.1 t)$.
\subsubsection{$X$-$Z$ dynamics}

Since the value of $\omega_{\mathrm{o}}$ is very small, to avoid being
overwhelmed by large numbers and similar to \cite{zhou17tcst}, we define a new
state $x=[X,\frac{\dot{X}}{\omega_{\mathrm{o}}},Z,\frac{\dot{Z}}%
{\omega_{\mathrm{o}}}]^{\mathrm{T}}$. Then the $X$-$Z$ dynamics in
(\ref{sys11}) can be rewritten as (\ref{sys_2}) with
\[ A=\omega_{\mathrm{o}}\left[ \begin{array} [c]{cccc}%
0 & 1 & 0 & 0\\ 0 & 0 & 0 & 2\\ 0 & 0 & 0 & 1\\ 0 & -2 & 3 & 0 \end{array} \right] , B=\frac{1}{\omega_{\mathrm{o}}}\left[ \begin{array} [c]{cc}%
0 & 0\\ 1 & 0\\ 0 & 0\\ 0 & 1 \end{array} \right] , C= \left[ \begin{array} [c]{cc}%
1 & 0\\ 0 & 0\\ 0 & 1\\ 0 & 0 \end{array} \right] ^{\mathrm{T}}.
\]

We can calculate that $r=2$ and $W_{\mathrm{o}}$ is invertible, namely,
Assumption \ref{ass2} is satisfied. For Theorems \ref{the_2} and \ref{sample_data}, direct calculation from
(\ref{L_h_def}) and (\ref{P_def}) gives
$\det(W_{\mathrm{o}}+L(d))
=\det(W_{\mathrm{o}}+P(d,t))
=(3\omega_{\mathrm{o}}d\sin(\omega_{\mathrm{o}}d)
+8\cos(\omega_{\mathrm{o}}d)-8)/d^{2}$.
For Theorem \ref{the4},
$\det(\mathcal W_{\mathrm{o}}(d,t))
=(3\omega_{\mathrm{o}}\Delta(d,t)\sin(\omega_{\mathrm{o}}\Delta(d,t))
+8\cos(\omega_{\mathrm{o}}\Delta(d,t))-8)/d^{2}$,
where $\Delta(d,t)=d+\tau(t-d-h)-\tau(t-h)$.
The matrices $W_{\mathrm{o}}+L(d)$ and $W_{\mathrm{o}}+P(d,t)$ are invertible for all $d>0$ except at isolated singular values, with the first positive one given by $d=2\pi/\omega_{\mathrm{o}}\approx86400$. For $\mathcal W_{\mathrm{o}}(d,t)$, the first positive singularity occurs when
$\Delta(d,t)=2\pi/\omega_{\mathrm{o}}$. For $\tau(t)=\sin^{2}(0.1t)$, the first uniformly nonsingular interval is approximately $0<d<86399.52$, obtained from $d+|\sin(0.1d)|<2\pi/\omega_{\mathrm{o}}$. Physically, at these singular values the oscillatory modes are sampled one orbital period apart and return to the same orbital phase, making certain velocity-related state directions indistinguishable from the available outputs. Thus, for $X$-$Z$ dynamics, the three cases do not require sufficiently small $d$; it is only necessary to avoid the corresponding isolated singular values. The controllers $a_X$ and $a_Z$ are then designed according to Theorems \ref{the_2}, \ref{the4}, and \ref{sample_data}.
\subsubsection{$Y$ dynamics}

The $Y$ dynamics in (\ref{sys11}) can be rewritten as (\ref{sys_2}) with
$x=[Y,\dot{Y}]^{\mathrm{T}}$,
$A={\scriptstyle\left[\begin{smallmatrix}0&1\\-\omega_{\mathrm{o}}^{2}&0\end{smallmatrix}\right]}$,
$B=[0,1]^{\mathrm{T}}$, and $C=[1,0]$. We have $r=2$ and
$W_{\mathrm{o}}=I_{2}$, so Assumption \ref{ass2} holds. From
(\ref{L_h_def}) and (\ref{P_def}),
$\det(W_{\mathrm{o}}+L(d))
=\det(W_{\mathrm{o}}+P(d,t))
=\sin(\omega_{\mathrm{o}}d)/(\omega_{\mathrm{o}}d)$.
For Theorem \ref{the4},
$\det(\mathcal W_{\mathrm{o}}(d,t))
=\sin(\omega_{\mathrm{o}}\Delta(d,t))/(\omega_{\mathrm{o}}d)$,
where $\Delta(d,t)=d+\tau(t-d-h)-\tau(t-h)$. Hence, the first positive
singularity of the former two matrices is
$d=\pi/\omega_{\mathrm{o}}\approx43200$, whereas, for
$\tau(t)=\sin^{2}(0.1t)$, the first interval of $d$ over which
$\mathcal W_{\mathrm{o}}(d,t)$ remains nonsingular for all $t\geq2h$ is
approximately $0<d<43199.75$, determined
by $d+|\sin(0.1d)|<\pi/\omega_{\mathrm{o}}$. At the singular values, the
oscillatory mode is sampled half an orbital period apart, so that
$C\mathrm{e}^{-Ad}=(-1)^{k}C$ and the two output samples contain linearly
dependent state information. Thus, sufficiently small $d$ is unnecessary;
one only needs to avoid the corresponding isolated singularities. The
controller $a_Y$ is then designed according to Theorems \ref{the_2},
\ref{the4}, and \ref{sample_data}, respectively.

Similar to \cite{zhou22tac}, take
$\chi(0)=[X(0),Y(0),Z(0),\dot X(0),\dot Y(0),$ $\dot Z(0)]
=[-1000,1000,1000,2,-2,2]$,
$S(t)=\sin^{2}(\pi t/h)$, and $h=3d$. Consider the noisy measurement
$y(t)=[X(t),Z(t),Y(t)]^{\mathrm T}+w(t)$, where
$w(t)=10^{j}v(t)$, $j=1,2,3$,
$v(t)=0.01(\mathrm{rnd}(3)+\sin(\omega t))$,
$\omega=[11,12,13]^{\mathrm T}$, and $\mathrm{rnd}(3)$ denotes a
three-dimensional vector whose entries are independently and uniformly
distributed over $[0,1]$, as in \cite{zimenko25auto}. For $d=50,60,70$, the corresponding convergence times are
$T=3h=9d=450,540,630$, respectively. For all three values of $d$, the choice $h=3d$ satisfies the
corresponding delay conditions in Theorems~\ref{the_2}, \ref{the4}, and \ref{sample_data}. Figures~\ref{figure3} and
\ref{figure10} show that, in the presence of measurement noise,
all three controllers drive the state into small ultimate
neighborhoods of the origin after the corresponding convergence
times. In particular, under $w(t)=100v(t)$ and $1000v(t)$,
Theorem~\ref{the_2} yields bounds of about $0.8$ and $8.8$, respectively,
while under $w(t)=10v(t)$ and $100v(t)$, Theorem~\ref{sample_data} yields bounds
of about $2.5$ and $25$, respectively. Figure~\ref{figure10} further shows
that the controller of Theorem~\ref{the4} achieves smaller bounds of about
$0.1$ and $0.8$ under
$w(t)=10v(t)$ and $100v(t)$. These results indicate that the proposed delayed
output-feedback controllers exhibit robustness against measurement noise
and can partially attenuate its oscillatory component without an explicit
filter.
%\begin{figure}[t]
%\centering \includegraphics[scale=0.5]{figure_2}\caption{State trajectories of
%the spacecraft rendezvous control system without measurement noises}%
%\label{figure2}%
%\end{figure}

\begin{figure}[t]
\centering \includegraphics[scale=0.5]{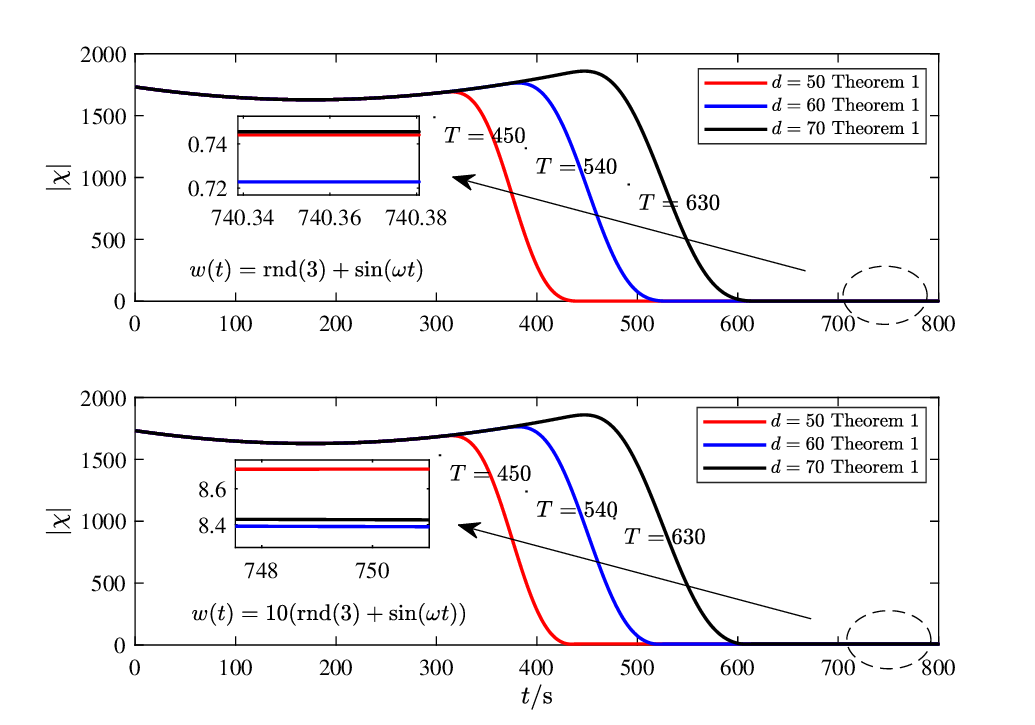}\caption{State trajectories of
the spacecraft rendezvous control system under measurement noises and without delays}%
\label{figure3}%
\end{figure}
\begin{figure}[t]
\centering \includegraphics[scale=0.5]{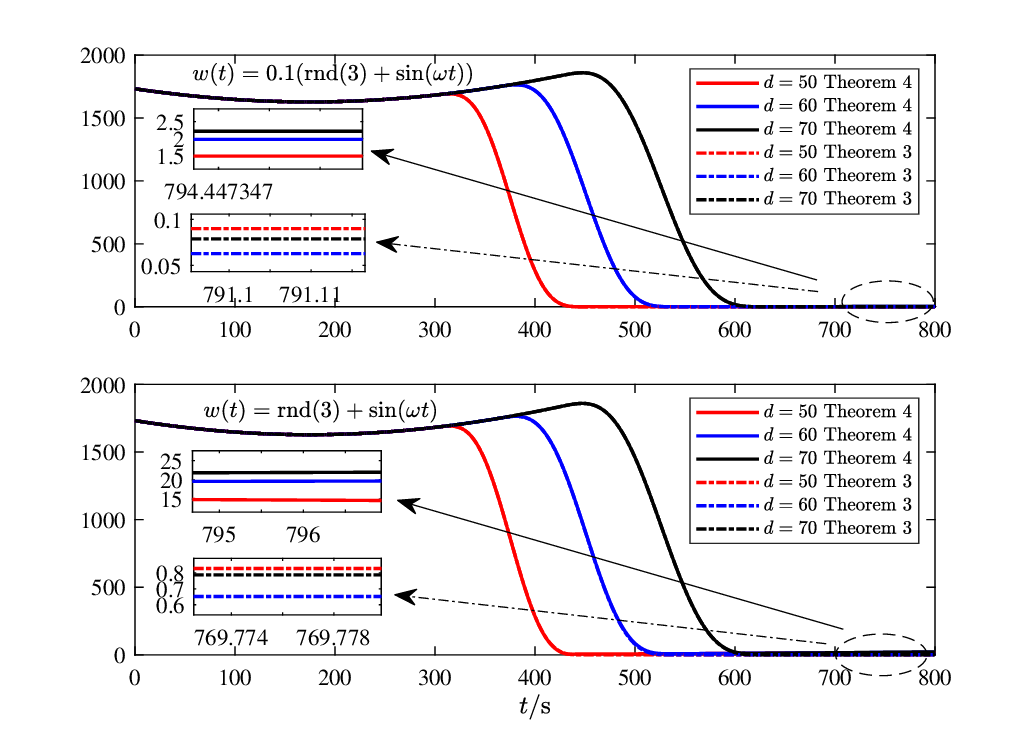}\caption{State trajectories of
the spacecraft rendezvous control system under measurement noises with delay $\tau(t)=\sin^2(0.1t)$ (Theorem \ref{the4}) and sampled-data measurements $t_{k}=kd$ (Theorem \ref{sample_data})}%
\label{figure10}%
\end{figure}
\section{\label{sec5}Conclusion}

This paper introduced two artificial delays with distinct roles. The first
realizes static output feedback using delayed output measurements, while the
second achieves prescribed/fixed-time stabilization. Furthermore, we
solve the fixed-time static output-feedback control problem in the presence
of output measurement delays. For systems without output measurement delays,
a time-invariant control gain may also be designed using the optimization
approach in \cite{wim20scl}; however, the existence of a feasible solution is
not guaranteed. An extension to fast-varying input and output delays (without any constraints on the delay derivatives) may be a topic for future research.

\end{document}